\documentclass[a4paper,10.9pt,reqno]{amsart}

\usepackage[T1]{fontenc}			
\usepackage[utf8]{inputenc}
\usepackage[italian, english]{babel}
\usepackage{amsmath, amsfonts, amssymb, amsthm, amscd, mathtools}
\usepackage{geometry}
\usepackage[square,numbers]{natbib}	
\usepackage{bbm, bm}
\usepackage{xspace} 
\usepackage{enumerate, enumitem}
\usepackage{hyperref}
\usepackage{xcolor}

\newcommand{\bbE}{{\ensuremath{\mathbb E}} }
\newcommand{\bbF}{{\ensuremath{\mathbb F}} }

\newcommand{\bbP}{{\ensuremath{\mathbb P}} }

\newcommand{\bbY}{{\ensuremath{\mathbb Y}} }

\newcommand{\cA}{{\ensuremath{\mathcal A}} }

\newcommand{\cF}{{\ensuremath{\mathcal F}} }

\newcommand{\cM}{{\ensuremath{\mathcal M}} }
\newcommand{\cN}{{\ensuremath{\mathcal N}} }

\newcommand{\cP}{{\ensuremath{\mathcal P}} }

\newcommand{\cY}{{\ensuremath{\mathcal Y}} }

\newcommand{\dd}{{\ensuremath{\mathrm d}} }
\newcommand{\de}{{\ensuremath{\mathrm e}} }

\newcommand{\dC}{{\ensuremath{\mathrm C}} }
\newcommand{\dD}{{\ensuremath{\mathrm D}} }

\newcommand{\dL}{{\ensuremath{\mathrm L}} }

\newcommand{\R}{\mathbb{R}}

\newcommand{\N}{\mathbb{N}}

\newcommand{\ind}{\ensuremath{\mathbf{1}}}
\newcommand{\one}{\ensuremath{\mathsf{1}}}
\DeclarePairedDelimiter{\abs}{\lvert}{\rvert}
\DeclarePairedDelimiter{\norm}{\lVert}{\rVert}
\DeclarePairedDelimiterX{\inprod}[2]{\langle}{\rangle}{#1, #2}
\DeclareMathOperator{\tr}{tr}

\renewcommand{\epsilon}{\varepsilon}
\newcommand{\changetheta}
{
\let\temp\theta
\let\theta\vartheta
\let\vartheta\temp
}
\newcommand{\changephi}
{
\let\temp\phi
\let\phi\varphi
\let\varphi\temp
}

\theoremstyle{plain}
\newtheorem{theorem}{Theorem}[section]
\newtheorem*{theorem*}{Theorem}
\newtheorem{lemma}[theorem]{Lemma}
\newtheorem*{lemma*}{Lemma}
\newtheorem{proposition}[theorem]{Proposition}
\newtheorem*{proposition*}{Proposition}

\theoremstyle{definition}

\newtheorem{assumption}{Assumption}[section]

\theoremstyle{remark}
\newtheorem{remark}{Remark}[section]
\newtheorem*{remark*}{Remark}

\definecolor{darkviolet}{rgb}{0.58, 0.0, 0.83}

\definecolor{gre}{rgb}{0.03,0.50,0.03}

\numberwithin{equation}{section}

\hypersetup{pdftitle={Nonlinear Filtering of Partially Observed Systems arising in Singular Stochastic Optimal Control},pdfauthor={Alessandro Calvia, Giorgio Ferrari},hidelinks,%
pdfstartview={XYZ null null 1.4},colorlinks=true,linkcolor=blue,urlcolor=blue}

\newcommand{\mail}[1]{\href{mailto:#1}{\normalfont\texttt{#1}}}

\makeatletter
\def\@setthanks{\vspace{-\baselineskip}\def\thanks##1{\@par##1\@addpunct.}\thankses}
\makeatother

\title[Filtering of Singularly Controlled Systems]{Nonlinear Filtering of Partially Observed Systems arising in Singular Stochastic Optimal Control}
\author[A.~Calvia]{Alessandro~Calvia}
\author[G.~Ferrari]{Giorgio~Ferrari}
\date{}
\thanks{\noindent A.~Calvia
\\
LUISS University, Department of Economics and Finance, Viale Romania 32, 00197 Rome (Italy).
\\
E-mail: \mail{acalvia@luiss.it}.
\medskip
\\
G.~Ferrari
\\
Bielefeld University, Center for Mathematical Economics (IMW), Universit\"atstrasse 25, 33615, Bielefeld (Germany).
\\
E-mail: \mail{giorgio.ferrari@uni-bielefeld.de}.
\medskip
\\
This research was supported by the 2019 INdAM-GNAMPA project \foreignlanguage{italian}{\textit{Problemi di controllo ottimo stocastico con osservazione parziale in dimensione infinita}}, of which the first author was Principal Investigator. Financial support by the German Research Foundation (DFG) through the Collaborative Research Centre 1283 is also gratefully acknowledged by the second author.
\medskip
}

\begin{document}

\changephi
\changetheta
\allowdisplaybreaks
	
\begin{abstract}
This paper deals with a nonlinear filtering problem in which a multi-dimensional signal process is additively affected by a process $\nu$ whose components have paths of bounded variation. The presence of the process $\nu$ prevents from directly applying classical results and novel estimates need to be derived. By making use of the so-called reference probability measure approach, we derive the Zakai equation satisfied by the unnormalized filtering process, and then we deduce the corresponding Kushner-Stratonovich equation. Under the condition that the jump times of the process $\nu$ do not accumulate over the considered time horizon, we show that the unnormalized filtering process is the unique solution to the Zakai equation, in the class of \mbox{measure-valued} processes having a square-integrable density. Our analysis paves the way to the study of stochastic control problems where a decision maker can exert singular controls in order to adjust the dynamics of an unobservable It\^o-process.
\end{abstract}
\maketitle
	
\noindent \textbf{Keywords:} Stochastic filtering; singularly controlled systems; reference probability measure; Zakai equation; Kushner-Stratonovich equation.

\smallskip

\noindent \textbf{AMS 2020:} 93E11, 60G35, 60H15, 60J25, 60J76.

\smallskip

\bigskip

\section{Introduction} 
\label{sec:intro}

This paper studies a stochastic filtering problem on a finite time horizon $[0,T]$, $T > 0$, in which the dynamics of a multi-dimensional process $X = (X_t)_{t \in [0,T]}$, called \emph{signal} or \emph{unobserved process}, are additively affected by a process having components of bounded variation.
The aim is to estimate the hidden state $X_t$, at each time $t \in [0,T]$, using the information provided by a further stochastic process $Y = (Y_t)_{t \in [0,T]}$, called \emph{observed process}; said otherwise, we look for the conditional distribution of $X_t$ given the available observation up to time $t$. This leads to derive an evolution equation for the filtering process, which is a probability measure-valued process satisfying, for any given bounded and measurable function $\phi \colon \R^m \to \R$,
\begin{equation*}
\pi_t(\phi) \coloneqq \int_{\R^m} \phi(x) \, \pi_t(\dd x) = \bbE\bigl[\phi(X_t) \bigm| \cY_t \bigr], \quad t \in [0,T],
\end{equation*}
where $(\cY_t)_{t \in [0,T]}$ is the natural filtration generated by $Y$ and augmented by $\bbP$-null sets.
The process $\pi$ provides the best estimate (in the usual $\dL^2$ sense) of the signal process $X$, given the available information obtained through the process $Y$.

Stochastic filtering is nowadays a well-established research topic. The literature on the subject is vast and many different applications have been studied: the reader may find a fairly detailed historical account in the book by \citet{bain:fundofstochfilt}. Classic references are the books by \citet{bensoussan:stochcontrol, kallianpur1980:stochfilt, liptsershiryaev2001:statistics} (cf. also \citet[Chapter~4]{bremaud:pp} for stochastic filtering with point process observation); more recent monographs are, e.g., the aforementioned book by \citet{bain:fundofstochfilt, crisanrozovski2011}, and \citet{xiong:intrtostochfiltth} (see also \citet[Chapter~22]{cohen:stochcalculus}).
Recently, different cases where the signal and/or the observation processes can have discontinuous trajectories (as in the present work) have been studied and explicit filtering equations have been derived: see, for instance, \citet{bandini2021:filt, calvia:filtcontrol, ceci2000filtering, ceci2001nonlinear, cecicolaneri:ks, cecicolaneri:zakai, confortola:filt, grigelionis2011}.

The main motivation of our analysis stems from the study of singular stochastic control problems under partial observation. 
Consider a continuous-time stochastic system whose position or level $X_t$ at time $t \in [0,T]$ is subject to random disturbances and can be adjusted instantaneously through (cumulative) actions that, as functions of time, do not have to be absolutely continuous with respect to Lebesgue measure. In particular, they may present a Cantor-like component and/or a jump component. The use of such singular control policies is nowadays common in applications in Economics, Finance, Operations Research, as well as in Mathematical Biology. Typical examples are, amongst others, (ir)reversible investment choices (e.g., \citet{RiedelSu}), dividends' payout (e.g., \citet{Reppen-etal}), inventory management problems (e.g., \citet{HarrisonTaksar}), as well as harvesting issues (e.g., \citet{AlvarezShepp}).
Suppose also that the decision maker acting on the system is not able to observe the dynamics of the controlled process $X$, but she/he can only follow the evolution of a noisy process $Y$, whose drift is a function of the signal process.
Mathematically, we assume that the pair $(X,Y)$ is defined on a filtered complete probability space $(\Omega, \cF, \bbF \coloneqq (\cF_t)_{t \in [0,T]}, \bbP)$ and that its dynamics are given, for any $t \in [0,T]$, by the following system of SDEs:
\begin{equation}\label{eq:model-intro}
\left\{
\begin{aligned}
\dd X_t &= b(t, X_t) \, \dd t + \sigma(t, X_t) \, \dd W_t + \dd \nu_t, & &X_{0^-} \sim \xi \in \mathcal{P}(\R^m), \\
\dd Y_t &= h(t, X_t) \, \dd t + \gamma(t) \, \dd B_t, & &Y_0 = y \in \R^n.
\end{aligned}
\right.
\end{equation}
Here: $\xi$ is a given probability distribution on $\R^m$; $W$ and $B$ are two independent $\bbF$-standard Brownian motions; coefficients $b, \sigma, h, \gamma$ are suitable measurable functions; $\nu$ is a c\`adl\`ag, $\R^m$-valued process with (components of) bounded variation, that is adapted to the previously introduced observation filtration $(\cY_t)_{t \in [0,T]}$.

Clearly, the decision maker might want to adjust the dynamics of $X$ in order to optimize a given performance criterion. Since $X$ is unobservable, this leads to a stochastic optimal control problem under partial observation, which can be tackled by deriving and studying the so-called \emph{separated problem}, an equivalent problem under full information (see, e.g., \citet{bensoussan:stochcontrol}), where the signal $X$ is formally replaced by its estimate provided by the filtering process $\pi$. However, to effectively solve the original optimization problem by means of the separated one, a first necessary step concerns the detailed study of the associated filtering problem. 

%

To the best of our knowledge, the derivation of explicit filtering equations in the setting described above has not yet received attention in the literature. In this paper we provide a first contribution in this direction. Indeed, the recent literature treating singular stochastic control problems under partial observation assumes that the observed process, rather than the signal one, is additively controlled (cf.\ \citet{Callegaro-etal}, \citet{DeAngelis}, \citet{DecampsVilleneuve}, and \citet{Federico-etal}). Clearly, such a modeling feature leads to a filtering analysis that is completely different from ours.

By making use of the so-called reference probability measure approach, we derive the \emph{Zakai stochastic partial differential equation} (SPDE) satisfied by the so-called \emph{unnormalized filtering process}, which is a measure-valued process, associated with the filtering process via a suitable change of probability measure. Then, we deduce the corresponding evolution equation for $\pi$, namely, the so-called \emph{Kushner-Stratonovich equation} or \emph{Fujisaki-Kallianpur-Kunita equation}. Furthermore, we show that the unnormalized filtering process is the unique solution to the Zakai equation, in the class of measure-valued processes having a square-integrable density. The latter result is proved under the technical requirement that the jump times of the  process $\nu$ affecting $X$ in \eqref{eq:model-intro} do not accumulate over the considered time-horizon. Although such a condition clearly poses a restriction on the generality of the model, we also acknowledge that it is typically satisfied by optimal control processes arising in singular stochastic control problems. 
It is important to notice that establishing conditions under which the unnormalized filtering process possesses a density paves the way to recast the separated problem as a stochastic control problem in a Hilbert space, as we will briefly explain in the next section.

The rest of the introduction is now devoted to a discussion of our approach and results at a more technical level.

\subsection{Methodology and main results}
\label{sec:results}

In this paper we are going to study the filtering problem described above through the so-called \emph{reference probability approach}, that we briefly summarize here. To start, let us notice that the model introduced in \eqref{eq:model-intro} is somewhat ill-posed. In fact, the dynamics of the signal process $X$ depend on the $(\cY_t)_{t \in [0,T]}$-adapted process $\nu$ while, simultaneously, the dynamics of the observed process $Y$ depend on $X$. Otherwise said, it is not clear how to define $\nu$, which has to be given \emph{a priori}, and \emph{circularity} arises if one attempts to introduce the partially observed system $(X,Y)$ as in \eqref{eq:model-intro}.

A possible way out of this impasse is to define $Y$ as a given Gaussian process independent of $X$ (see \eqref{eq:Ydef}). In this way, it makes sense to fix a $(\cY_t)_{t \in [0,T]}$-adapted process $\nu$ and to define the dynamics of the signal process $X$ as in the first SDE of \eqref{eq:model-intro} (see also \eqref{eq:XSDE}). Finally, under suitable assumptions, there exists a probability measure change (cf. \eqref{eq:girsanov}) that allows us to recover the dynamics of $Y$ as in the second SDE of \eqref{eq:model-intro} (see also \eqref{eq:YSDEctrl}). It is important to notice that the resulting probability depends on the initial law $\xi$ of $X_{0^-}$ and on $\nu$. 
To derive the associated Kushner-Stratonovich equation there are two main approaches in the literature: The \emph{Innovations approach} and the aforementioned reference probability approach. Although it might be possible to derive the filtering dynamics in our context by using the former approach, we follow the latter method. 

Our first main results is Theorem~\ref{th:Zakai}, where we deduce the Zakai equation verified by the unnormalized filtering process (see \eqref{eq:rho} for its definition). From this result, as a byproduct, we deduce in Theorem~\ref{th:KS} the Kushner-Stratonovich equation satisfied by the filtering process. It is worth noticing that, given the presence of the bounded-variation process $\nu$ in the dynamics of $X$, Theorem~\ref{th:Zakai} cannot be obtained by invoking classical results, but novel estimates need to be derived (cf.\ Lemma~\ref{lem:gronwall} and Proposition~\ref{prop:etamart}). In particular, we employ a change of variable formula for Lebesgue-Stieltjes integrals.


It is clear that in applications, for instance to optimal control problems, establishing uniqueness of the solution to the Zakai equation or to the Kushner-Stratonovich equation is essential. In the literature there are several approaches to tackle this problem, most notably the following four: The filtered martingale problem approach, originally proposed by \citet{kurtzocone88:filt}, and later extended to singular martingale problems in \citep{kurtzstockbridge2001:MP} (see also \citep{kurtznappo2011:FMP}); the PDE approach, as in the book by \citet{bensoussan:stochcontrol} (see also \citep[Section~4.1]{bain:fundofstochfilt}); the functional analytic approach, introduced by \citet{lucicheunis2001:uniqueness} (see also \citep[Section~4.2]{bain:fundofstochfilt}); the density approach, studied in \citet{kurtzxiong1999:SPDEs} (see also \citep[Section~7]{bain:fundofstochfilt} and \citep{xiong:intrtostochfiltth}).

The first three methods allow to prove uniqueness of the solution to the Zakai equation in a suitable class of measure-valued processes. However, they  do not guarantee that the unique measure-valued process solution to the Zakai equation admits a density process, a fact that has an impact on the study of the separated problem. Indeed, without requiring or establishing conditions guaranteeing existence of such a density process, the separated problem must be formulated in an appropriate Banach space of measures and, as a consequence, the \emph{Hamilton-Jacobi-Bellman} (HJB) equation associated to the separated problem must be formulated in such a general setting as well. As a matter of fact, only recently some techniques have been developed to treat this case, predominantly in the theory of mean-field games (an application to optimal control problems with partial observation is given in \citep{bandini2019:wasserstein}).

%

A more common approach in the literature considers, instead, the density process as the state variable for the separated problem. If it is possible to show that such a density process is the unique solution of a suitable SPDE in $\dL^2(\R^m)$, the so-called \emph{Duncan-Mortensen-Zakai} equation, then this $\dL^2(\R^m)$-valued process can be equivalently used as state variable in the separated problem. This is particularly convenient, since for optimal control problems in Hilbert spaces a well-developed theory is available, at least in the regular case (see, e.g., the monograph by \citet{fabbri:soc}).
Therefore, in view of possible future applications to singular optimal control problems under partial observation, we adopted the density approach to prove that, under suitable assumptions, the unnormalized filtering process is the unique solution to the Zakai equation in the class of measure-valued processes admitting a density with respect to Lebesgue measure.

We show this result, first, in the case where $\nu$ is a continuous process (cf.\ Theorem~\ref{th:Zakaiuniqnucont}) and, then, in the case where the jump times of $\nu$ do not accumulate in the time interval $[0,T]$ (see Theorem~\ref{th:Zakaiuniqnujump}). As we already observed, although this assumption prevents to achieve full generality, it has a clear interpretation and it is usually satisfied by the examples considered in the literature.
From a technical side, it seems that a direct approach using the method proposed by \citep{kurtzxiong1999:SPDEs} is not feasible to treat the case of accumulating jumps, due to difficulties in estimating crucial quantities in the arguments used, that are related to the jump component of filtering process. A possible workaround might consists in approximating the process $\nu$ by cutting away jumps of size smaller than some $\delta > 0$ and then, provided that a suitable tightness property holds, pass to the limit, as $\delta \to 0$, in the relevant equations. However, this is a delicate and lengthy reasoning, which is left for future research.

%

\medskip

The rest of this paper is organized as follows. Section~\ref{sec:notation} provides notation used throughout this work. Section~\ref{sec:model} introduces the filtering problem. The Zakai and Kushner-Stratonovich equations are then derived in Section~\ref{sec:Zakai}, while the uniqueness of the solution to the Zakai equation is proved in Section~\ref{sec:uniqueness}. Finally, Appendix \ref{app:technical} collects the proof of technical results.

\subsection{Notation}
\label{sec:notation}

In this section we collect the main notation used in this work.
Throughout the paper the set $\N$ denotes the set of natural integers $\N = \{1, 2, \dots \}$, $\N_0 = \{0, 1, \dots \}$, and $\R$ is the set of real numbers. 

For any $m \times n$ matrix $A = (a_{ij})$, the symbol $A^*$ denotes its transpose and $\norm{A}$ is its Frobenius norm; i.e., $\norm{A} = (\sum_{i=1}^m \sum_{j=1}^n a_{ij}^2)^{1/2}$.
For any $x,y\in \R^d$, $\norm{x}$ denotes the Euclidean norm of $x$ and $x \cdot y = x^* y$ indicates the inner product of $x$ and $y$. 
For a fixed Hilbert space $H$, we denote its inner product by $\langle \cdot, \cdot \rangle$ and by $\norm{\cdot}_H$ its norm.

The symbol $\ind_C$ denotes the indicator function of a set $C$, while $\one$ is the constant function equal to $1$. The symbol $\int_a^b$ denotes $\int_{[a,b]}$ for any $-\infty < a \leq b < +\infty$.

For any $d \in \N$ and $T > 0$, we denote by $\dC^{1,2}_b([0,T] \times \R^d)$ the set of real-valued bounded measurable functions on $[0,T] \times \R^d$, that are continuously differentiable once with respect to the first variable and twice with respect to the second, with bounded derivatives. 
For any such function, the symbol $\partial_t$ denotes the derivative with respect to the first variable, while $\dD_x = (\partial_1, \dots, \partial_d)$ and $\dD^2_x = (\partial^2_{ij})_{i,j = 1}^d$ denote, respectively, the gradient and the Hessian matrix with respect to the second variable. 
Furthermore, we simply write $\dC^{2}_b(\R^d)$, when we are considering a real-valued bounded function on $\R^d$ that is twice continuously differentiable with bounded derivatives.

For any $d \in \N$ we indicate by $\dL^2(\R^d)$ the set of all square-integrable functions with respect to Lebesgue measure and for all $k \in \N$ we denote by $W^2_k(\R^d)$ the Sobolev space of all functions $f \in \dL^2(\R^d)$ such that the partial derivatives $\partial^\alpha$ exist in the weak sense and are in $\dL^2(\R^d)$, whenever the multi-index $\alpha = (\alpha_1, \dots, \alpha_d)$ is such that $\alpha_1 + \cdots + \alpha_d \leq k$. 

For a fixed metric space $E$, endowed with the Borel $\sigma$-algebra, we denote by $\cP(E)$, $\cM_+(E)$, and $\cM(E)$ the sets of probability, finite positive, and finite signed measures on $E$, respectively. If $\mu \in \cM(E)$, then $\abs{\mu} \in \cM_+(E)$ is the total variation of $\mu$.

For any given c\`adl\`ag stochastic process $Z = (Z_t)_{t \geq 0}$ defined on a probability space $(\Omega, \cF, \bbP)$, we denote by $(Z_{t^-})_{t \geq 0}$ the left-continuous version of $Z$ (i.e., $Z_{t^-} = \lim_{s \to t^-} Z_s, \, \bbP$-a.s., for any $t \geq 0$), and by $\Delta Z_t \coloneqq Z_t - Z_{t^-}$ the jump of $Z$ at time $t \geq 0$. If $Z$ has finite variation over $[0,t]$, for all $t \geq 0$, $\abs{Z}$ (resp.\ $Z^+$, $Z^-$) is the \emph{variation process} (resp.\ the positive part process, the negative part process) of $Z$, i.e., the process such that, for each $t \in [0,T]$ and $\omega \in \Omega$, $\abs{Z}_t(\omega)$ (resp.\ $Z^+_t(\omega)$, $Z^-_t(\omega)$) is the total variation (resp.\ the positive part, the negative part) of the function $s \mapsto Z_s(\omega)$ on $[0,t]$. It is useful to remember that $Z = Z^+ - Z^-$, $\abs{Z} = Z^+ + Z^-$, and that $Z^+$, $Z^-$ are non-decreasing processes.

Finally, with the word \emph{measurable} we refer to \emph{Borel-measurable}, unless otherwise specified.

%
%
%


\section{Model formulation}
\label{sec:model}

Let $T > 0$ be a given fixed time horizon and $(\Omega, \cF, \bbF \coloneqq (\cF_t)_{t \in [0,T]}, \bbP)$ be a complete filtered probability space, with $\bbF$ satisfying the usual assumptions.

Define on $(\Omega, \cF, \bbF, \bbP)$ two independent $\bbF$-adapted standard Brownian motions $W$ and $\overline B$, taking values in $\R^d$ and $\R^n$, respectively, with $d, n \in \N$. Let then $\gamma \colon [0,T] \to \R^{n \times n}$ be a measurable function such that, for each $t \in [0,T]$, $\gamma(t)$ is symmetric, with $\gamma_{ij}(t) \in \dL^2([0,T])$, for all $i, j = 1, \dots, n$, and uniformly positive definite; that is, there exists $\delta > 0$ such that for all $t \in [0,T]$ and all $x \in \R^m$
\begin{equation}\label{eq:gammaunifpd}
\gamma(t) x \cdot x \geq \delta \norm{x}^2.
\end{equation}
These requirements guarantee in particular that the \emph{observed process} $Y = (Y_t)_{t \in [0,T]}$, defined as
\begin{equation}\label{eq:Ydef}
Y_t = y + \int_0^t \gamma(t) \, \dd \overline B_t, \quad t \in [0,T], \, y \in \R^n, 
\end{equation}
is an $\R^n$-valued $\bbF$-adapted martingale, of which we take a continuous version. Clearly, it holds
\begin{equation}
\label{eq:YSDE}
\dd Y_t = \gamma(t) \, \dd \overline B_t, \quad t \in [0,T], \qquad Y_0 = y \in \R^n.
\end{equation}

\begin{remark}
It is not restrictive to require that $\gamma$ is symmetric (and uniformly positive definite). Indeed, suppose that $\overline B$ is an $\R^k$-valued $\bbF$-adapted standard Brownian motion and that $\gamma \colon [0,T] \to \R^{n \times k}$ is such that $\gamma \gamma^*(t):=\gamma(t) \gamma^*(t)$ is uniformly positive definite. Then, we can obtain an equivalent model defining the $\R^n$-valued $\bbF$-adapted standard Brownian motion $\widetilde B = (\widetilde B_t)_{t \in [0,T]}$ through:
\begin{equation*}
\dd \widetilde B_t \coloneqq \bigl(\gamma \gamma^*(t)\bigr)^{-1/2} \, \gamma(t) \, \dd \overline B_t, \quad t \in [0,T].
\end{equation*} 
In fact, in this case~\eqref{eq:YSDE} becomes:
\begin{equation*}
\dd Y_t = \bigl(\gamma \gamma^*(t)\bigr)^{1/2} \, \dd \widetilde B_t, \quad t \in [0,T], \qquad Y_0 = y \in \R^n,
\end{equation*}
and clearly $\bigl(\gamma \gamma^*(t)\bigr)^{1/2}$ is symmetric (and uniformly positive definite).
\end{remark}

We indicate with the symbol $\bbY$ the completed natural filtration generated by $Y$, i.e., $\bbY \coloneqq (\cY_t)_{t \in [0,T]}$, with $\cY_t \coloneqq \{Y_s \colon 0 \leq s \leq t\} \lor \cN$, where $\cN$ is the collection of all $\bbP$-null sets. 

\begin{remark}\label{rem:filtrY}
Notice that since $\gamma$ is invertible, $\bbY$ coincides with the completed natural filtration generated by $\overline B$ and is, therefore, right-continuous. These facts will be useful in the sequel.
\end{remark}

Next, we consider a probability distribution $\xi$ on $\R^m$; measurable functions $b \colon [0,T] \times \R^m \to \R^m$ and $\sigma \colon [0,T] \times \R^m \to \R^{m \times d}$, with $m \in \N$; a $\bbY$-adapted, c\`adl\`ag, $\R^m$-valued process $\nu$ whose components have paths of finite variation. We introduce the following requirements, that will be in force throughout the paper.

\newpage

\begin{assumption}\label{hyp:main}
\mbox{}
\begin{enumerate}[label=(\roman*)]
\item There exist constants $C_b$ and $L_b$ such that for all $t \in [0,T]$
\begin{equation}\label{eq:blip}
\norm{b(t,x) - b(t,x')} \leq L_b \norm{x - x'} \quad \text{and} \quad \norm{b(t,0)} \leq C_b, \quad \forall x, x' \in \R^m.
\end{equation}

\item There exist constants $C_\sigma$ and $L_\sigma$ such that for all $t \in [0,T]$
\begin{equation}\label{eq:sigmalip}
\norm{\sigma(t,x) - \sigma(t,x')} \leq L_\sigma \norm{x - x'} \quad \text{and} \quad \norm{\sigma(t,0)} \leq C_\sigma, \quad \forall x, x' \in \R^m.
\end{equation}

\item The probability law $\xi \in \cP(\R^m)$ satisfies
\begin{equation}\label{eq:xisq}
\int_{\R^m} \norm{x}^2 \, \xi(\dd x) < +\infty.
\end{equation}

\item\label{hyp:nu} The $\R^m$-valued process $\nu$ is $\bbY$-adapted, c\`adl\`ag, with $\nu_{0^-} = 0$. Its components have paths of finite variation, which in particular satisfy
\begin{equation}\label{eq:nufinitefuel}
\abs{\nu^i}_T \leq K, \qquad \forall i = 1, \dots, m,
\end{equation}
for some constant $K > 0$.
\end{enumerate}
\end{assumption}

Under Assumption \ref{hyp:main}, for any such $\nu$, the following SDE for the signal process $X = (X_t)_{t \in [0,T]}$ admits a unique strong solution:
\begin{equation}\label{eq:XSDE}
\dd X_t = b(t, X_t) \, \dd t + \sigma(t, X_t) \, \dd W_t + \dd \nu_t, \quad t \in [0,T], \qquad X_{0^-} \sim \xi \in \cP(\R^m).
\end{equation}
It is important to bear in mind, especially in applications to optimal control problems, that the solution to \eqref{eq:XSDE} and all the quantities that are related to it depend on the the probability distribution $\xi$ and on $\nu$. However, for the ease of exposition, we will not stress this dependence in the sequel.

\begin{remark}\label{rem:XSDE}
Conditions \eqref{eq:blip} and \eqref{eq:sigmalip} ensure that SDE \eqref{eq:XSDE} admits a unique strong solution for any $\nu$. If we assume, in addition, that \eqref{eq:xisq} and \eqref{eq:nufinitefuel} hold, then we have that, for some constant $\kappa$ depending on $T$, $b$, $\sigma$, and $\nu$,
\begin{equation}\label{eq:Xestimate}
\bbE[\sup_{t \in [0,T]} \norm{X_t}^2] \leq \kappa(1+ \bbE[\norm{X_{0^-}}^2]) < +\infty,
\end{equation}
since $\bbE[\norm{X_{0^-}}^2] = \int_{\R^m} \norm{x}^2 \, \xi(\dd x)$. Proofs of these statements are standard and can be found, for instance, in \citep{cohen:stochcalculus,protter2004}.
\end{remark}

We finally arrive to the model we intend to analyze \emph{via} a change of measure. Let $h \colon [0,T] \times \R^m \to \R^n$ be a measurable function satisfying the following condition, that will stand from now on. 

\begin{assumption}\label{hyp:h}
There exists a constant $C_h$ such that for all $t \in [0,T]$
\begin{equation}\label{eq:hlin}
\norm{h(t,x)} \leq C_h(1+\norm{x}), \quad \forall x \in \R^m.
\end{equation}
\end{assumption}
For all $t \in [0,T]$ define then:
\begin{equation}\label{eq:eta}
\eta_t \coloneqq \exp\left\{\int_0^t \gamma^{-1}(s) h(s, X_s) \, \dd \overline B_s - \dfrac 12 \int_0^t \norm{\gamma^{-1}(s) h(s, X_s)}^2 \, \dd s\right\}.
\end{equation}

By Proposition~\ref{prop:etamart}, $\eta$ is a $(\bbP,\bbF)$-martingale, under Assumptions~\ref{hyp:main} and~\ref{hyp:h}.
Therefore, we can introduce the probability measure $\widetilde \bbP$ on $(\Omega, \cF_T)$ satisfying%

\begin{equation}\label{eq:girsanov}
\frac{\dd \widetilde \bbP}{\dd \bbP} \bigg|_{\cF_T} = \eta_T.
\end{equation}
By Girsanov's Theorem, the process $B = (B_t)_{t \in [0,T]}$ given by $B_t \coloneqq \overline B_t - \int_0^t \gamma^{-1}(s) h(s, X_s) \, \dd s$, $t \in [0,T]$,
is a $(\widetilde \bbP,\bbF)$-Brownian motion, and under $\widetilde \bbP$ the dynamics of the observed process are provided by the SDE:
\begin{equation}
\label{eq:YSDEctrl}
\dd Y_t = h(t, X_t) \, \dd t + \gamma(t) \, \dd B_t, \quad t \in [0,T], \qquad Y_0 = y \in \R^n.
\end{equation}

We see that equations~\eqref{eq:XSDE} and~\eqref{eq:YSDEctrl} are formally equivalent to model~\eqref{eq:model-intro}. Observe, however, that the Brownian motion driving~\eqref{eq:YSDEctrl} is not a source of noise given \emph{a priori}, but it is obtained through a probability measure change; moreover, our construction implies that it depends on the initial law $\xi$ and on process $\nu$. This formulation is typical in optimal control problems under partial observation (see, e.g., \citep[Chapter~8]{bensoussan:stochcontrol}) and has the advantage of avoiding the circularity problem discussed in the Introduction.

\begin{remark}\label{rem:nufinitefuel}
If the partially observed system defined by \eqref{eq:XSDE} and \eqref{eq:YSDEctrl} describes the state variables of a singular optimal control problem, where $\nu$ is the control process, then condition \eqref{eq:nufinitefuel} implies that the singular control is of \emph{finite fuel} type (see \citet{elkarouikaratzas, ocone-etal} for early contributions).
\end{remark}

\begin{remark}
It is worth noticing that all the results in this paper remain valid if we allow $b$ to depend also on $\omega$, as long as the map $(\omega,t) \mapsto b(\omega,t,x)$ is $\bbY$-adapted and c\`adl\`ag, for each $x \in \R^m$, and condition \eqref{eq:blip} holds uniformly with respect to $\omega$ (i.e., $L_b$ and $C_b$ do not depend on $\omega$). To extend our subsequent results to this case, it suffices to apply the so-called \emph{freezing lemma} whenever necessary.

This modeling flexibility is important when it comes to treating controlled dynamics where $b$ is a deterministic function, depending on an additional parameter representing the action of a regular control $\alpha = (\alpha_t)_{t \in [0,T]}$. Clearly, this control must be c\`adl\`ag and $\bbY$-adapted, i.e., based on the available information. The measurability requirement above ensures that the map $(\omega, t) \mapsto b(t,x,\alpha_t(\omega))$ is $\bbY$-adapted.
\end{remark}


\section{The Zakai and Kushner-Stratonovich equations}
\label{sec:Zakai}

In this section we will deduce the Zakai equation satisfied by the unnormalized filtering process, defined in~\eqref{eq:rho}. As a byproduct, we will deduce the Kushner-Stratonovich equation satisfied by the filtering process (see~\eqref{eq:filter} for its definition). As anticipated in the Introduction, we will use the reference probability approach to achieve these results. The reference probability will be precisely $\bbP$, under which the observed process is Gaussian and satisfies \eqref{eq:Ydef}. However, the probability measure that matters from a modelization point of view is $\widetilde \bbP$, which defined in \eqref{eq:girsanov}. Indeed, we will define the filtering process under this measure. It is important to bear in mind that $\widetilde \bbP$ and $\bbP$ are equivalent probability measures. Hence, any result holding $\bbP$-a.s., holds also $\widetilde \bbP$-a.s., and we will write only the first of these two wordings.

The following technical lemma is needed. Its proof is a consequence of the facts highlighted in Remark~\ref{rem:filtrY} and it is omitted (the reader may refer, for instance, to~\citep[Prop. 3.15]{bain:fundofstochfilt}). In what follows we will denote $\cY \coloneqq \cY_T$.
\begin{lemma}\label{lem:condexp}
Let $Z$ be an $\cF_t$-measurable, $\bbP$-integrable random variable, $t \in [0,T]$. Then
\begin{equation*}
\bbE[Z \mid \cY_t] = \bbE[Z \mid \cY].
\end{equation*} 
\end{lemma}

As previously anticipated, the \emph{filtering process} $\pi = (\pi_t)_{t \in [0,T]}$ is a $\cP(\R^m)$-valued process providing the conditional law of the signal $X$ at each time $t \in [0,T]$, given the available observation up to time $t$. It is defined for any bounded and measurable $\phi \colon [0,T] \times \R^m \to \R$ as:
\begin{equation}\label{eq:filter}
\pi_t(\phi_t) \coloneqq \widetilde \bbE\bigl[\phi(t, X_t) \bigm| \cY_t \bigr], \quad t \in [0,T],
\end{equation}
where $\phi_t(x) \coloneqq \phi(t,x)$, for any $(t,x) \in [0,T] \times \R^m$.
Since $\R^m$ is a complete and separable metric space, $\pi$ is a well-defined, $\cP(\R^m)$-valued and $\bbY$-adapted process.\footnote{Without any particular assumptions on $\bbY$, the filtering process is adapted with respect to the right-continuous enlargement of $\bbY$. However, as previously observed, in our model $\bbY$ is already right-continuous.} Moreover, $\pi$ admits a c\`adl\`ag modification, since $X$ is c\`adl\`ag (see, e.g. \citep[Cor. 2.26]{bain:fundofstochfilt}). Hence, in the sequel we shall consider $\pi$ as a $\bbY$-progressively measurable process.

We recall the useful Kallianpur-Striebel formula, which holds thanks to Proposition~\ref{prop:etamart} for any bounded and measurable $\phi \colon [0,T] \times \R^m \to \R$ and for any fixed $t \in [0,T]$ (for a proof see, e.g., \citep[Prop. 3.16]{bain:fundofstochfilt})
\begin{equation}\label{eq:KallianpurStriebel}
\pi_t(\phi_t) = \frac{\bbE\bigl[\eta_t \phi(t, X_t) \bigm| \cY \bigr]}{\bbE\bigl[\eta_t \bigm| \cY \bigr]}, \quad \bbP\text{-a.s.}
\end{equation}

This formula allows us to define the measure-valued process $\rho = (\rho_t)_{t \in [0,T]}$, called \emph{unnormalized} conditional distribution of $X$, or \emph{unnormalized filtering process}, defined, for any bounded and measurable $\phi \colon [0,T] \times \R^m \to \R$, as:
\begin{equation}\label{eq:rho}
\rho_t(\phi_t) \coloneqq \bbE\bigl[\eta_t \phi(t, X_t) \bigm| \cY_t \bigr], \quad t \in [0,T].
\end{equation}
Given the properties of $\pi$ and of $\eta$ it is possible to show (see, e.g., \citep[Lemma~3.18]{bain:fundofstochfilt}) that $\rho$ is c\`adl\`ag and $\bbY$-adapted, hence $\bbY$-progressively measurable. Moreover, the Kallianpur-Striebel formula implies that for any bounded and measurable $\phi \colon [0,T] \times \R^m \to \R$ and for any fixed $t \in [0,T]$:
\begin{equation}\label{eq:rhopilink}
\pi_t(\phi_t) = \frac{\rho_t(\phi_t)}{\rho_t(\one)}, \quad \bbP\text{-a.s.},
\end{equation}
where $\one \colon \R^m \to \R$ is the constant function equal to $1$.

To describe the local dynamics of the signal process $X$, let us introduce the operator $\cA$, defined for any $\phi \in \dC^{1,2}_b([0,T] \times \R^m)$ as:
\begin{equation}\label{eq:operatorA}
\cA \phi(t,x) \coloneqq \dD_x \phi(t,x) \cdot b(t,x) + \frac 12 \tr\bigl(\dD^2_x \phi(t,x) \, \sigma\sigma^*(t,x)\bigr), \quad (t,x) \in [0,T] \times \R^m.
\end{equation}
We can also define the family of operators $\cA_t$, $t \in [0,T]$, given by:
\begin{equation*}
\cA_t \phi(x) = \dD_x \phi(x) \cdot b(t,x) + \frac 12 \tr\bigl(\dD^2_x \phi(x) \, \sigma\sigma^*(t,x)\bigr), \quad x \in \R^m, \, \phi \in \dC^2_b(\R^m).
\end{equation*}

To obtain the Zakai equation we need, first, to write the semimartingale decomposition of the process $\bigl(\phi(t,X_t)\bigr)_{t \in [0,T]}$. For any $\phi \in \dC^{1,2}_b([0,T] \times \R^m)$ we have, applying It\^o's formula:
\begin{multline}\label{eq:phisemimart}
\phi(t,X_t) = \phi(0, X_{0^-}) + \int_0^t \bigl[\partial_s + \cA\bigr] \phi(s, X_s) \, \dd s + \int_0^t \dD_x \phi(s, X_{s^-}) \, \dd \nu_s
\\
+ \sum_{0 \leq s \leq t} \Bigl[\phi(s,X_s) - \phi(s, X_{s^-}) - \dD_x \phi(s, X_{s^-}) \cdot \Delta \nu_s\Bigr] + M_t^\phi, \quad t \in [0,T].
\end{multline}
Here, $M_t^\phi \coloneqq \int_0^t \dD_x \phi(t,X_t) \, \sigma(t, X_t) \, \dd W_t$, $t \in [0,T]$, is a square-integrable $(\bbP, \bbF)$-martingale, thanks to conditions \eqref{eq:blip} and \eqref{eq:sigmalip} (see also Remark \ref{rem:XSDE}).

We need the following two technical Lemmata. Up to minor modifications, their proofs follow that of~\citep[Lemma~3.21]{bain:fundofstochfilt}.
\begin{lemma}\label{lem:exchangeB}
Let $\Psi = (\Psi_t)_{t \in [0,T]}$ be a real-valued $(\bbP, \bbF)$-progressively measurable process such that
\begin{equation*}
\bbE\biggl[\int_0^T \Psi_s^2 \, \dd s \biggr] < +\infty.
\end{equation*}
Then, for any $j = 1, \dots, k$ we have
\begin{equation*}
\bbE\biggl[\int_0^t \Psi_s \, \dd \overline B_s^j \biggm| \cY\biggr] = \int_0^t \bbE[\Psi_s \mid \cY] \, \dd \overline B_s^j, \quad t \in [0,T].
\end{equation*}
\end{lemma}

\begin{lemma}\label{lem:exchangeMphi}
Let $\Psi = (\Psi_t)_{t \in [0,T]}$ be a real-valued $(\bbP, \bbF)$-progressively measurable process satisfying\footnote{If $M$ is any $(\bbP, \bbF)$-square integrable martingale, $\langle M \rangle$ denotes its $(\bbP, \bbF)$-predictable quadratic variation.}
\begin{equation*}
\bbE\biggl[\int_0^T \Psi_s^2 \, \dd \langle M^\phi \rangle_s \biggr] < +\infty.
\end{equation*}
Then,
\begin{equation*}
\bbE\biggl[\int_0^t \Psi_s \, \dd M_s^\phi \biggm| \cY\biggr] = 0, \quad t \in [0,T].
\end{equation*}
\end{lemma}

We are now ready to state the main result of this section, namely, to provide the Zakai equation.
\begin{theorem}\label{th:Zakai}
Suppose that Assumptions \ref{hyp:main} and \ref{hyp:h} are satisfied and, moreover, that 
\begin{equation}\label{eq:xi3}
\int_{\R^m} \norm{x}^3 \, \xi(\dd x) < +\infty.
\end{equation}
Then, for any $\phi \in \dC_b^{1,2}([0,T] \times \R^m)$, the unnormalized conditional distribution $\rho$ satisfies the Zakai equation:
\begin{multline}\label{eq:Zakai}
\rho_t(\phi_t) = \xi(\phi_0) + \int_0^t \rho_s\bigl(\bigl[\partial_s + \cA_s\bigr] \phi_s\bigr) \, \dd s + \int_0^t \rho_{s^-}\bigl(\dD_x \phi_s\bigr) \, \dd \nu_s + \int_0^t \gamma^{-1}(s) \rho_s(\phi_s h_s) \, \dd \overline B_s
\\
+\sum_{0 \leq s \leq t} \Bigl[\rho_{s^-}\bigl(\phi_s(\cdot + \Delta \nu_s) - \phi_s - \dD_x \phi_s \cdot \Delta \nu_s\bigr)\Bigr], \quad \bbP\text{-a.s.}, \quad t \in [0,T],
\end{multline}
where $\xi(\phi_0) \coloneqq \int_{\R^m} \phi(0, x) \, \xi(\dd x)$ and, for all $t \in [0,T]$, $h_t(\cdot) \coloneqq h(t,\cdot)$,
\begin{align*}
\int_0^t \rho_{s^-}\bigl(\dD_x \phi_s\bigr) \, \dd \nu_s &\coloneqq \sum_{i=1}^m \int_0^t \rho_{s^-}\bigl(\partial_i \phi_s\bigr) \, \dd \nu^i_s,
\\
\int_0^t \gamma^{-1}(s) \rho_s(\phi_s h_s) \, \dd \overline B_s &\coloneqq \sum_{i=1}^n \sum_{j=1}^n \int_0^t \gamma^{-1}_{ij}(s) \rho_s(\phi_s h^j_s) \, \dd \overline B^i_s.
\end{align*}
\end{theorem}

\begin{proof}
Fix $t \in [0,T]$ and $\phi \in \dC_b^{1,2}([0,T] \times \R^m)$. Let us introduce the constants
\begin{align*}
C_\phi &\coloneqq \sup_{t,x} \abs{\phi(t,x)},
&
C'_\phi &\coloneqq \sup_{t,x} \norm{\dD_x \phi(t,x)},
&
C''_\phi &\coloneqq \sup_{t,x} \norm{\dD^2_x \phi(t,x)},
\end{align*}
where the suprema are taken over $[0,T] \times \R^m$. The proof is organized in several steps.

\smallskip

\noindent \textbf{Step 1. (Approximation)} For any fixed $\epsilon > 0$, define the bounded process $\eta^\epsilon = (\eta^\epsilon_t)_{t \in [0,T]}$:
\begin{equation}\label{eq:etaepsilon}
\eta^\epsilon_t \coloneqq \frac{\eta_t}{1+ \epsilon \eta_t}, \quad t \in [0,T],
\end{equation}
where $\eta$ is defined in~\eqref{eq:eta}.
Both $\eta$ and $\eta^\epsilon$ have continuous trajectories and this fact will be used in what follows without further mention.

Applying It\^o's formula we obtain
\begin{equation*}
\eta^\epsilon_t = \dfrac{1}{1+\epsilon} - \int_0^t \dfrac{\epsilon \eta_s^2}{(1+\epsilon \eta_s)^3} \norm{\gamma^{-1}(s) h(s, X_s)}^2 \, \dd s + \int_0^t \dfrac{\eta_s}{(1+\epsilon\eta_s)^2} \gamma^{-1}(s) h(s, X_s) \, \dd \overline B_s.
\end{equation*}

Denoting by $[\cdot, \cdot]$ the optional quadratic covariation operator, thanks to the integration by parts rule and recalling~\eqref{eq:phisemimart} we get
\begin{align}
&\mathop{\phantom{=}} \eta^\epsilon_t \phi(t,X_t)
= \dfrac{\phi(0,X_{0^-})}{1+\epsilon} + \int_0^t \eta^\epsilon_{s^-} \, \dd \phi(s,X_s) + \int_0^t \phi(s,X_{s^-}) \, \dd \eta^\epsilon_s + \int_0^t \dd \bigl[\eta^\epsilon, \phi(\cdot, X) \bigr]_s
\nonumber
\\
&= \dfrac{\phi(0,X_{0^-})}{1+\epsilon} + \int_0^t \eta^\epsilon_{s^-} \bigl[\partial_s + \cA\bigr] \phi(s, X_s) \, \dd s + \int_0^t \eta^\epsilon_{s^-} \dD_x \phi(s, X_{s^-}) \, \dd \nu_s
\nonumber
\\
&\quad + \sum_{0 \leq s \leq t} \eta^\epsilon_{s^-} \Bigl[\phi(s,X_s) - \phi(s, X_{s^-}) - \dD_x \phi(s, X_{s^-}) \, \Delta \nu_s\Bigr] + \int_0^t \eta^\epsilon_{s^-} \, \dd M_s^\phi
\nonumber
\\
&\quad - \int_0^t \dfrac{\epsilon \eta_s^2 \phi(s,X_{s^-})}{(1+\epsilon \eta_s)^3} \norm{\gamma^{-1}(s) h(s, X_s)}^2  \, \dd s + \int_0^t \dfrac{\eta_s \phi(s,X_{s^-})}{(1+\epsilon\eta_s)^2} \gamma^{-1}(s) h(s, X_s) \, \dd \overline B_s.
\label{eq:diffprod}
\end{align}

\smallskip

\noindent \textbf{Step 2. (Projection onto $\bbY$)} Notice that $X_t = X_{t^-} + \Delta \nu_t$, $\bbP$-a.s., $t \in [0,T]$, and that, since $\cY_{0^-} = \cY_0 = \{\emptyset, \Omega\}$, we have
\begin{equation*}
\bbE[\phi(0, X_{0^-}) \mid \cY_{0^-} ] = \int_{\R^m} \phi(0, x) \, \xi(\dd x) = \xi(\phi_0).
\end{equation*}
Therefore, taking conditional expectation with respect to $\cY$, we have (rearranging some terms)
\begin{align}
&\mathrel{\phantom{=}}\bbE[\eta^\epsilon_t \phi(t,X_t) \mid \cY ] = \dfrac{\xi(\phi_0)}{1+\epsilon} + \bbE\biggl[\int_0^t \eta^\epsilon_{s^-} \bigl[\partial_s + \cA\bigr] \phi(s, X_s) \, \dd s \biggm| \cY \biggr] 
\nonumber
\\
&+ \bbE\biggl[\int_0^t\eta^\epsilon_{s^-} \dD_x \phi(s, X_{s^-}) \, \dd \nu_s \biggm| \cY \biggr] + \bbE\biggl[\int_0^t\dfrac{\eta_s \phi(s,X_{s^-})}{(1+\epsilon\eta_s)^2} \gamma^{-1}(s) h(s, X_s) \, \dd \overline B_s \biggm| \cY \biggr]
\nonumber
\\
&+ \bbE\biggl[\sum_{0 \leq s \leq t} \eta^\epsilon_{s^-} \Bigl[\phi(s,X_{s^-} + \Delta \nu_s) - \phi(s, X_{s^-}) - \dD_x \phi(s, X_{s^-}) \cdot \Delta \nu_s\Bigr] \biggm| \cY \biggr]
\nonumber
\\
&+ \bbE\biggl[\int_0^t \eta^\epsilon_{s^-} \, \dd M_s^\phi \biggm| \cY \biggr] - \bbE\biggl[\int_0^t \dfrac{\epsilon \eta_s^2 \phi(s,X_{s^-})}{(1+\epsilon \eta_s)^3} \norm{\gamma^{-1}(s) h(s, X_s)}^2  \, \dd s \biggm| \cY \biggr].
\label{eq:semimartproj}
\end{align}
We analyze now each of the terms appearing in~\eqref{eq:semimartproj}. For any bounded $\cY$-measurable $Z$, thanks to conditions \eqref{eq:blip} and \eqref{eq:sigmalip}, there exists a constant $C_1$, depending on $Z$, $\epsilon$, $\phi$, $b$, and $\sigma$ such that
\begin{equation*}
|Z \eta^\epsilon_t \bigl[\partial_t + \cA\bigr] \phi(t, X_t)| \leq C_1(1+ \norm{X_t}^2), \quad t \in [0,T],
\end{equation*}
which implies, using the estimate given in \eqref{eq:Xestimate},
\begin{equation*}
\bbE\biggl[\int_0^t Z \eta^\epsilon_s \bigl[\partial_s + \cA\bigr] \phi(s, X_s) \, \dd s\biggr]
\leq C \bbE\left[\int_0^t (1+\norm{X_s}^2) \, \dd s\right] \leq C_1T[1+\kappa(1+ \bbE[\norm{X_{0^-}}^2])] < +\infty.
\end{equation*}
Therefore, applying the tower rule and Fubini-Tonelli's theorem,
\begin{equation*}
\bbE\biggl[Z \, \bbE\biggl[\int_0^t \eta^\epsilon_s \bigl[\partial_s + \cA\bigr] \phi(s, X_s) \, \dd s \biggm| \cY \biggr]\biggr]
= \bbE\biggl[Z \int_0^t \bbE[\eta^\epsilon_s \bigl[\partial_s + \cA\bigr] \phi(s, X_s) \mid \cY] \, \dd s \biggr],
\end{equation*}
whence
\begin{equation}
\label{eq:ev1}
\bbE\biggl[\int_0^t \eta^\epsilon_s \bigl[\partial_s + \cA\bigr] \phi(s, X_s) \, \dd s \biggm| \cY \biggr] = \int_0^t \bbE[\eta^\epsilon_s \bigl[\partial_s + \cA\bigr] \phi(s, X_s) \mid \cY] \, \dd s.
\end{equation}
Similarly, for any bounded $\cY$-measurable $Z$ we have that
\begin{equation*}
\norm{Z \eta^\epsilon_t \dD_x \phi(t, X_{t^-})} \leq \dfrac{\abs{Z}C'_\phi}{\epsilon} < +\infty, \quad \dd \bbP \otimes \dd t\text{-a.e.}
\end{equation*}
This fact will allow to use Fubini-Tonelli's theorem in formula \eqref{eq:condexpectintegrnu} below.
We need to introduce the changes of time associated to the processes $\nu^{i,+}$ and $\nu^{i,-}$, $i = 1, \dots, m$, defined as
\begin{equation*}
C^{i,+}_t \coloneqq \inf\{s \geq 0 \colon \nu^{i,+}_s \geq t\}, \quad C^{i,-}_t \coloneqq \inf\{s \geq 0 \colon \nu^{i,-}_s \geq t\}, \quad t \geq 0, \quad i = 1, \dots,m,
\end{equation*}
where $\nu^{i,+}$ (resp.\ $\nu^{i,-}$) denotes the positive part (resp.\ negative part) process of the $i$-th component of process $\nu$ (see the list of notation in Section~\ref{sec:notation} for a more detailed definition). 

For each $t \geq 0$ and $i = 1, \dots, m$, $C^{i,+}_t$ and $C^{i,-}_t$ are $\bbY$-stopping times (see, e.g., \citep[Chapter~VI, Def. 56]{dellacheriemeyer:B} or \citep[Proposition~I.1.28]{jacod2013:limit}). Hence, applying the change of time formula (see, e.g., \citep[Chapter~VI, Equation~(55.1)]{dellacheriemeyer:B} or \citep[Equation~(1), p. 29]{jacod2013:limit}) and Fubini-Tonelli's theorem, we get
\begin{align}
&\mathop{\phantom{=}} \bbE\biggl[Z \, \bbE\biggl[\int_0^t\eta^\epsilon_s \dD_x \phi(s, X_{s^-}) \, \dd \nu_s \biggm| \cY \biggr]\biggr]
= \bbE\biggl[\int_0^{+\infty} \ind_{s \leq t} Z \eta^\epsilon_s \dD_x \phi(s, X_{s^-}) \, \dd \nu_s\biggr]
\nonumber
\\
&= \sum_{i=1}^m \bbE\biggl[\int_0^{+\infty} \ind_{s \leq t} Z \eta^\epsilon_s \partial_i \phi(s, X_{s^-}) \, \dd \nu^{i,+}_s\biggr] - \sum_{i=1}^m \bbE\biggl[\int_0^{+\infty} \ind_{s \leq t} Z \eta^\epsilon_s \partial_i \phi(s, X_{s^-}) \, \dd \nu^{i,-}_s\biggr] \notag
\\
&= \sum_{i=1}^m \bbE\biggl[\int_0^{+\infty} \ind_{C_s^{i,+} \leq t} Z \eta^\epsilon_{C_s^{i,+}} \partial_i \phi(C_s^{i,+}, X_{({C_s^{i,+}})^-}) \ind_{C_s^{i,+} < +\infty} \, \dd s\biggr] \notag
\\
&\qquad - \sum_{i=1}^m \bbE\biggl[\int_0^{+\infty} \ind_{C_s^{i,-} \leq t} Z \eta^\epsilon_{C_s^{i,-}} \partial_i \phi(C_s^{i,-}, X_{({C_s^{i,-}})^-}) \ind_{C_s^{i,-} < +\infty} \, \dd s\biggr] \notag
\\
&= \sum_{i=1}^m \int_0^{+\infty} \bbE\left[\ind_{C_s^{i,+} \leq t} Z \bbE\bigl[\eta^\epsilon_{C_s^{i,+}} \partial_i \phi(C_s^{i,+}, X_{({C_s^{i,+}})^-}) \bigm| \cY\bigr] \ind_{C_s^{i,+} < +\infty}\right] \, \dd s\biggr] \notag
\\
&\qquad - \sum_{i=1}^m \int_0^{+\infty} \bbE\left[\ind_{C_s^{i,-} \leq t} Z \bbE\bigl[\eta^\epsilon_{C_s^{i,-}} \partial_i \phi(C_s^{i,-}, X_{({C_s^{i,-}})^-}) \bigm| \cY\bigr] \ind_{C_s^{i,-} < +\infty}\right] \, \dd s\biggr] \notag
\\
&= \sum_{i=1}^m \bbE\biggl[\int_0^{+\infty} \! \ind_{s \leq t} Z \bbE[\eta^\epsilon_s \, \partial_i \phi(s, X_{s^-}) \mid \cY] \, \dd \nu^{i,+}_s\biggr] \!-\! \sum_{i=1}^m \bbE\biggl[\int_0^{+\infty} \! \ind_{s \leq t} Z \bbE[\eta^\epsilon_s \, \partial_i \phi(s, X_{s^-}) \mid \cY] \, \dd \nu^{i,-}_s\biggr] \notag
\\
&= \bbE\biggl[\int_0^{+\infty} \ind_{s \leq t} Z \bbE[\eta^\epsilon_s \, \dD_x \phi(s, X_{s^-}) \mid \cY] \, \dd \nu_s\biggr]
= \bbE\biggl[Z \int_0^t \bbE[\eta^\epsilon_s \, \dD_x \phi(s, X_{s^-}) \mid \cY] \, \dd \nu_s\biggr],
\label{eq:condexpectintegrnu}
\end{align}
whence
\begin{equation}
\label{eq:ev2}
\bbE\biggl[\int_0^t\eta^\epsilon_s \dD_x \phi(s, X_{s^-}) \, \dd \nu_s \biggm| \cY \biggr] = \int_0^t \bbE[\eta^\epsilon_s \, \dD_x \phi(s, X_{s^-}) \mid \cY] \, \dd \nu_s.
\end{equation}

Next, using \eqref{eq:gammainvhintegr} we obtain
\begin{equation*}
\bbE\biggl[\int_0^t \biggl(\frac{\eta^\epsilon_s}{1+\epsilon \eta_s} \phi(s,X_{s^-}) \norm{\gamma^{-1}(s) h(s, X_s)} \biggr)^2 \, \dd s \biggr] 
\leq \frac{C_\phi^2}{\epsilon^2} \, \bbE\biggl[\int_0^t \norm{\gamma^{-1}(s) h(s, X_s)}^2 \, \dd s \biggr] < +\infty,
\end{equation*}
hence, by Lemma~\ref{lem:exchangeB} we have:
\begin{equation}
\label{eq:ev3}
\bbE\biggl[\int_0^t \dfrac{\eta_s \phi(s,X_{s^-})}{(1+\epsilon\eta_s)^2} \gamma^{-1}(s) h(s, X_s) \, \dd \overline B_s \biggm| \cY \biggr] = \int_0^t \bbE\biggl[\dfrac{\eta_s \phi(s,X_{s^-})}{(1+\epsilon\eta_s)^2} \gamma^{-1}(s) h(s, X_s) \biggm| \cY \biggr] \, \dd \overline B_s.
\end{equation}

Recalling that $\abs{\nu}^i_T \leq K$, $\bbP$-a.s., and hence $\abs{\Delta \nu^i_t} \leq K$, for all $t \in [0,T]$ and all $i=1,\dots,m$, $\bbP$-a.s., for any bounded $\cY$-measurable $Z$ we have that
\begin{align*}
&\mathop{\phantom{\leq}} \sum_{0 \leq s \leq t} \bbE\biggl| Z \eta^\epsilon_s \Bigl[\phi(s,X_{s^-} + \Delta \nu_s) - \phi(s, X_{s^-}) - \dD_x \phi(s, X_{s^-}) \, \Delta \nu_s\Bigr] \biggr|
\leq \frac{\abs{Z}}{\epsilon} \frac{C''_\phi}{2} \sum_{0 \leq s \leq t} \bbE\bigl[\norm{\Delta \nu_s}^2\bigr]
\\
&= \frac{\abs{Z}}{\epsilon} \frac{C''_\phi}{2} \bbE\biggl[\sum_{0 \leq s \leq t} \Delta \nu_s^* \Delta \nu_s\biggr]
\leq \frac{\abs{Z}}{\epsilon} \frac{C''_\phi}{2} \sum_{i=1}^m \bbE\left[\int_0^t \abs{\Delta \nu^i_s} \dd \abs{\nu^i}_s\right]
\leq \frac{\abs{Z}}{\epsilon} \frac{C''_\phi}{2} mK^2 < + \infty.
\end{align*}
Therefore, using once more Fubini-Tonelli's theorem
\begin{align*}
&\mathrel{\phantom{=}}\bbE\biggl[Z \, \bbE\biggl[\sum_{0 \leq s \leq t} \eta^\epsilon_s \Bigl[\phi(s,X_{s^-} + \Delta \nu_s) - \phi(s, X_{s^-}) - \dD_x \phi(s, X_{s^-}) \cdot \Delta \nu_s\Bigr] \biggm| \cY \biggr]\biggr]
\\
&= \bbE\biggl[Z \sum_{0 \leq s \leq t} \bbE\Bigl[ \eta^\epsilon_s \Bigl[\phi(s,X_{s^-} + \Delta \nu_s) - \phi(s, X_{s^-}) - \dD_x \phi(s, X_{s^-}) \cdot \Delta \nu_s\Bigr] \Bigm| \cY\Bigr] \, \biggr],
\end{align*}
and hence
\begin{align}
\label{eq:ev4}
&\mathop{\phantom{=}} \bbE\biggl[\sum_{0 \leq s \leq t} \eta^\epsilon_s \Bigl[\phi(s,X_{s^-} + \Delta \nu_s) - \phi(s, X_{s^-}) - \dD_x \phi(s, X_{s^-}) \cdot \Delta \nu_s\Bigr] \biggm| \cY \biggr] \notag
\\
&= \sum_{0 \leq s \leq t} \bbE\Bigl[ \eta^\epsilon_s \Bigl[\phi(s,X_{s^-} + \Delta \nu_s) - \phi(s, X_{s^-}) - \dD_x \phi(s, X_{s^-}) \cdot \Delta \nu_s\Bigr] \Bigm| \cY\Bigr].
\end{align}

Finally, being $\eta^\epsilon$ bounded, Lemma~\ref{lem:exchangeMphi} entails $\bbE\bigl[\int_0^t \eta^\epsilon_s \, \dd M_s^\phi \bigm| \cY \bigr] = 0$,
and, using the same rationale of the previous evaluations,
\begin{equation}
\label{eq:ev5}
\bbE\biggl[\int_0^t \dfrac{\epsilon \eta_s^2 \phi(s,X_{s^-})}{(1+\epsilon \eta_s)^3} \norm{\gamma^{-1}(s) h(s, X_s)}^2 \, \dd s \biggm| \cY \biggr]
= \int_0^t \bbE\biggl[\dfrac{\epsilon \eta_s^2 \phi(s,X_{s^-})}{(1+\epsilon \eta_s)^3} \norm{\gamma^{-1}(s) h(s, X_s)}^2 \biggm| \cY\biggr] \, \dd s.
\end{equation}

Taking into account \eqref{eq:ev1}, \eqref{eq:ev2}, \eqref{eq:ev3}, \eqref{eq:ev4}, and \eqref{eq:ev5}, Equation~\eqref{eq:semimartproj} becomes
\begin{align}
&\mathrel{\phantom{=}}\bbE[\eta^\epsilon_t \phi(t,X_t) \mid \cY ] = \frac{\xi(\phi_0)}{1+\epsilon} + \int_0^t \bbE[\eta^\epsilon_s \bigl[\partial_s + \cA\bigr] \phi(s, X_s) \mid \cY] \, \dd s
\nonumber
\\
&- \int_0^t \bbE\biggl[\dfrac{\epsilon \eta_s^2 \phi(s,X_{s})}{(1+\epsilon \eta_s)^3} \norm{\gamma^{-1}(s) h(s, X_s)}^2 \biggm| \cY\biggr] \, \dd s \notag
\\
&+ \int_0^t \bbE\biggl[\dfrac{\eta_s \phi(s,X_{s})}{(1+\epsilon\eta_s)^2} \gamma^{-1}(s) h(s, X_s) \biggm| \cY \biggr] \, \dd \overline B_s + \int_0^t \bbE[\eta^\epsilon_s \, \dD_x \phi(s, X_{s^-}) \mid \cY] \, \dd \nu_s
\nonumber
\\
&+ \sum_{0 \leq s \leq t} \bbE\Bigl[ \eta^\epsilon_s \Bigl[\phi(s,X_{s^-} + \Delta \nu_s) - \phi(s, X_{s^-}) - \dD_x \phi(s, X_{s^-}) \cdot \Delta \nu_s\Bigr] \Bigm| \cY\Bigr].
\label{eq:semimartproj2}
\end{align}

\smallskip

\noindent \textbf{Step 3. (Taking limits)} It remains to show that all the terms appearing in \eqref{eq:semimartproj2} converge appropriately to give \eqref{eq:Zakai}. As $\epsilon \to 0$, we have that $\eta^\epsilon_t \to \eta_t$, $\bbE[\eta^\epsilon_t \phi(t,X_t) \mid \cY ] \longrightarrow \rho_t(\phi)$, for all $t \in [0,T]$, and
\begin{equation*}
\bbE[\eta^\epsilon_t \bigl[\partial_t + \cA\bigr] \phi(t, X_t) \mid \cY] \longrightarrow \rho_t\bigl(\bigl[\partial_t + \cA_t\bigr] \phi_t\bigr), \quad \dd\bbP\otimes \dd t\text{-a.e.}
\end{equation*}
Using boundedness of $\phi$ and \eqref{eq:blip}, \eqref{eq:sigmalip}, we get that
\begin{equation*}
\abs{\bbE[\eta^\epsilon_t \bigl[\partial_t + \cA\bigr] \phi(t, X_t) \mid \cY]} \leq C_2 \bbE[\eta_t(1+\norm{X_t}^2) \mid \cY], \quad t \in [0,T],
\end{equation*} 
for some constant $C_2$, depending on $\phi$, $b$, and $\sigma$. The r.h.s.\ of this inequality is $\dd \bbP \otimes \dd t$ integrable on $\Omega \times [0,t]$, since (apply again the tower rule and Fubini-Tonelli's theorem)
\begin{equation*}
\bbE\left[\int_0^t C_2 \bbE[\eta_s(1+\norm{X_s}^2) \mid \cY] \, \dd s\right] 
\leq C_2 T \bigl\{1 + \kappa(1+ \widetilde \bbE[\norm{X_{0^-}}^2])\bigr\} < +\infty,
\end{equation*}
where we used \eqref{eq:Xestimate} (which holds also under $\widetilde \bbP$ because the dynamics of $X$ does not change under this measure), and the fact that $\widetilde \bbE[\norm{X_{0^-}}^2] = \bbE[\norm{X_{0^-}}^2] < +\infty$.

Using the conditional form of the dominated convergence theorem, we have that, for all $t \in [0,T]$,
\begin{equation*}
\bbE\biggl[\int_0^t \bbE[\eta^\epsilon_s \bigl[\partial_s + \cA\bigr] \phi(s, X_s) \mid \cY] \, \dd s \biggm| \cY \biggr] \longrightarrow \bbE\biggl[\int_0^t \rho_s\bigl(\bigl[\partial_s + \cA_s\bigr] \phi_s\bigr) \, \dd s \biggm| \cY \biggr], \quad \bbP\text{-a.s.},
\end{equation*}
as $\epsilon \to 0$, whence, noticing that the integrals are $\cY$-measurable random variables,
\begin{equation*}
\int_0^t \bbE[\eta^\epsilon_s \bigl[\partial_s + \cA\bigr] \phi(s, X_s) \mid \cY] \, \dd s \longrightarrow \int_0^t \rho_s\bigl(\bigl[\partial_s + \cA_s\bigr] \phi_s\bigr) \, \dd s, \quad \bbP\text{-a.s.}, \quad \forall t \in [0,T].
\end{equation*}

We consider, now, the term on the second line of \eqref{eq:semimartproj2}. We have that, for all $t \in [0,T]$,
\begin{equation*}
\bbE\biggl[\dfrac{\epsilon \eta_t^2 \phi(t,X_{t})}{(1+\epsilon \eta_t)^3} \norm{\gamma^{-1}(t) h(t, X_t)}^2 \biggm| \cY\biggr] \longrightarrow 0,
\end{equation*}
as $\epsilon \to 0$, and that
\begin{equation*}
\bbE\biggl[\dfrac{\epsilon \eta_t^2 \phi(t,X_{t})}{(1+\epsilon \eta_t)^3} \norm{\gamma^{-1}(t) h(t, X_t)}^2 \biggm| \cY\biggr] 
\leq C_\phi \bbE[\eta_t\norm{\gamma^{-1}(t) h(t, X_t)}^2 \mid \cY]. 
\end{equation*}
The r.h.s.\ of the last inequality is $\dd \bbP \otimes \dd t$ integrable on $\Omega \times [0,t]$, since
\begin{equation*}
\bbE\left[\int_0^t C_\phi \bbE[\eta_s\norm{\gamma^{-1}(s) h(s, X_s)}^2 \mid \cY] \, \dd s\right]
\leq n C_\phi C_h C_\gamma T[1 + \kappa(1+ \widetilde \bbE[\norm{X_{0^-}}^2])] < +\infty,
\end{equation*}
where we used \eqref{eq:gammainvhintegr}, that holds also under $\widetilde \bbP$ (again, because the dynamics of $X$ does not change under this measure), and the fact that $\widetilde \bbE[\norm{X_{0^-}}^2] = \bbE[\norm{X_{0^-}}^2] < +\infty$.

Hence, reasoning as above, after applying the conditional form of the dominated convergence theorem we obtain that, for all $t \in [0,T]$, as $\epsilon \to 0$,
\begin{equation*}
\int_0^t \bbE\biggl[\dfrac{\epsilon \eta_s^2 \phi(s,X_{s})}{(1+\epsilon \eta_s)^3} \norm{\gamma^{-1}(s) h(s, X_s)}^2 \biggm| \cY\biggr] \, \dd s \longrightarrow 0, \quad \bbP\text{-a.s.}
\end{equation*}

Looking at the third line of \eqref{eq:semimartproj2}, the next step is to show that
\begin{equation*}
\int_0^t \bbE\biggl[\dfrac{\eta_s \phi(s,X_{s})}{(1+\epsilon\eta_s)^2} \gamma^{-1}(s) h(s, X_s) \biggm| \cY \biggr] \, \dd \overline B_s \longrightarrow \int_0^t \gamma^{-1}(s) \rho_s(\phi_s h_s) \, \dd \overline B_s, \quad \bbP\text{-a.s.}
\end{equation*}
The proof of this fact is standard (see, e.g., \citep[Theorem~3.24 and Exercise~3.25.i]{bain:fundofstochfilt} or \citep[Theorem~4.1.1]{bensoussan:stochcontrol}). It is important to notice that condition \eqref{eq:xi3} intervenes here.

Next, we examine the other integral in the third line of \eqref{eq:semimartproj2}. We have that
\begin{equation*}
\bbE[\eta^\epsilon_t \, \dD_x \phi(t, X_{t^-}) \mid \cY] \longrightarrow \rho_{t^-}\bigl(\dD_x \phi_t\bigr), \quad \bbP\text{-a.s.},
\end{equation*}
as $\epsilon \to 0$, for all $t \in [0,T]$. Notice that, for any $t \in [0,T]$,
\begin{equation*}
\norm{\bbE[\eta^\epsilon_t \, \dD_x \phi(t, X_{t^-}) \mid \cY]} \leq C'_\phi\bbE[\eta_t \mid \cY].
\end{equation*}
Since $\eta$ is non-negative, a $\bbY$-optional version of $\left\{\bbE[\ind_{t \leq T} \eta_t \mid \cY_t]\right\}_{t \geq 0}$ is given by the $\bbY$-optional projection of $\left\{\ind_{t \leq T} \eta_t\right\}_{t \geq 0}$ (see, e.g., \citep[Corollary~7.6.8]{cohen:stochcalculus}). Therefore, applying \citep[Chapter~VI, Theorem~57]{dellacheriemeyer:B}  and using Lemma~\ref{lem:condexp} we get that for all $t \in [0,T]$, and all $i = 1, \dots, m$,
\begin{equation*}
\bbE\biggl[\int_0^T C'_\phi \bbE[\eta_t \mid \cY] \, \dd \abs{\nu^i}_t\biggr] 
= C'_\phi \bbE\biggl[\int_0^{+\infty} \bbE[\ind_{t \leq T} \eta_t \mid \cY_t] \, \dd \abs{\nu^i}_t\biggr] 
= C'_\phi \bbE\biggl[\int_0^T \eta_t \, \dd \abs{\nu^i}_t\biggr] < +\infty,
\end{equation*}
where finiteness of $\bbE[\int_0^T \eta_t \, \dd \abs{\nu^i}_t]$ can be established with a reasoning analogous to the proof of \eqref{eq:etanuest2}.
Therefore, we can apply the conditional form of the dominated convergence theorem, to obtain that, for all $t \in [0,T]$, as $\epsilon \to 0$,
\begin{equation*}
\bbE\biggl[\int_0^t \bbE[\eta^\epsilon_s \, \dD_x \phi(s, X_{s^-}) \mid \cY] \, \dd \nu_s \biggm| \cY \biggr] \longrightarrow \bbE\biggl[\int_0^t \rho_{s^-}\bigl(\dD_x \phi_s\bigr) \, \dd \nu_s \biggm| \cY \biggr], \quad \bbP\text{-a.s.}
\end{equation*}
Since the integrals are $\cY$-measurable random variables, this implies that, for all $t \in [0,T]$, as $\epsilon \to 0$,
\begin{equation*}
\int_0^t \bbE[\eta^\epsilon_s \, \dD_x \phi(s, X_{s^-}) \mid \cY] \, \dd \nu_s \longrightarrow \int_0^t \rho_{s^-}\bigl(\dD_x \phi_s\bigr) \, \dd \nu_s, \quad \bbP\text{-a.s.}
\end{equation*}

Finally, looking at the fourth line of \eqref{eq:semimartproj2}, we have that, for all $t \in [0,T]$, as $\epsilon \to 0$,
\begin{multline*}
\Lambda^\epsilon_t \coloneqq \bbE\Bigl[ \eta^\epsilon_t \Bigl[\phi(t,X_{t^-} + \Delta \nu_t) - \phi(t, X_{t^-}) - \dD_x \phi(t, X_{t^-}) \cdot \Delta \nu_t\Bigr] \Bigm| \cY\Bigr] 
\\
\longrightarrow \rho_{t^-}\bigl(\phi_t(\cdot + \Delta \nu_t) - \phi_t - \dD_x \phi_t \cdot \Delta \nu_t\bigr), \quad \bbP\text{-a.s.}
\end{multline*}
Observe that, for any $t \in [0,T]$, $\Lambda^\epsilon_t$ is bounded by $\frac 12 C''_\phi \bbE[\eta_t \norm{\Delta \nu_t}^2 \mid \cY]$, which is positive and integrable with respect to the product of measure $\bbP$ and the jump measure associated to $\nu$, since:
\begin{align*}
&\mathop{\phantom{=}}\bbE\biggl[\sum_{0 \leq s \leq t} \frac 12 C''_\phi \bbE[\eta_s \norm{\Delta \nu_s}^2 \mid \cY] \biggr]
= \frac 12 C''_\phi \sum_{0 \leq s \leq t} \bbE[\eta_s \norm{\Delta \nu_s}^2] 
=\frac 12 C''_\phi \bbE\left[\sum_{0 \leq s \leq t} \eta_s \Delta \nu_s \cdot \Delta \nu_s\right]
\\
&\leq \frac 12 C''_\phi \sum_{i=1}^m \bbE\left[\int_0^t \eta_s \abs{\Delta \nu^i_s} \dd \abs{\nu^i}_s\right]
\leq \frac 12 C''_\phi K \sum_{i=1}^m \bbE\left[\int_0^t \eta_s \dd \abs{\nu^i}_s\right]
< +\infty.
\end{align*}
By the conditional form of the dominated convergence theorem, we have that, for all $t \in [0,T]$, as $\epsilon \to 0$,
\begin{equation*}
\bbE\Bigl[\sum_{0 \leq s \leq t} \Lambda^\epsilon_s \Bigm| \cY \Bigr] \longrightarrow \bbE\Bigl[\sum_{0 \leq s \leq t} \rho_{s^-}\bigl(\phi_s(\cdot + \Delta \nu_s) - \phi_s - \dD_x \phi_s \cdot \Delta \nu_s\bigr) \Bigm| \cY \Bigr], \quad \bbP\text{-a.s.}
\end{equation*}
and since the sums are $\cY$-measurable random variables, this implies that, for all $t \in [0,T]$, as $\epsilon \to 0$,
\begin{equation*}
\sum_{0 \leq s \leq t} \Lambda^\epsilon_s \longrightarrow \sum_{0 \leq s \leq t} \rho_{s^-}\bigl(\phi_s(\cdot + \Delta \nu_s) - \phi_s - \dD_x \phi_s \cdot \Delta \nu_s\bigr), \quad \bbP\text{-a.s.} \qedhere
\end{equation*}
\end{proof}

\begin{remark}\label{rem:ZakaiSPDE}
If the jump times of the process $\nu$ do not accumulate over $[0,T]$, then the Zakai equation can be split into successive linear SPDEs between the jumps of $\nu$ (i.e., of $X$). Set $T_0 = 0$, denote by $(T_n)_{n \in \N}$ the sequence of jump times of $\nu$ and indicate by $\nu^c$ the continuous part of $\nu$. Then, for any $\phi \in \dC^{1,2}_b([0,T] \times \R^m)$ and any $n \in \N_0$ we have $\bbP$-a.s.
\begin{equation}\label{eq:ZakaiSPDE}
\left\{
\begin{aligned}
&\dd \rho_t(\phi_t) \!=\! \rho_t\bigl(\bigl[\partial_t \!+\! \cA_t\bigr]\! \phi_t\bigr) \dd t + \rho_{t^-}\!\bigl(\dD_x \phi_t\bigr) \dd \nu^c_t 
+ \gamma^{-1}\!(t) \rho_t(\phi_t h_t) \dd \overline B_t, & & t \!\in\! [T_n \!\land\! T, T_{n+1} \!\land\! T),
\\
&\rho_{0^-}(\phi_0) = \xi(\phi_0),
\\
&\rho_{T_n}(\phi) = \rho_{{T_n}^-}\bigl(\phi_{T_n}(\cdot + \Delta \nu_{T_n})\bigr).
\end{aligned}
\right.
\end{equation}
\end{remark}

We are now ready to deduce, from the Zakai equation, the \emph{Kushner-Stratonovich} equation, i.e., the equation satisfied by the filtering process $\pi$, defined in~\eqref{eq:filter}.
The proof of the following two results follows essentially the same steps of~\citep[Lemma~3.29 and Theorem~3.30]{bain:fundofstochfilt}, up to necessary modifications due to the present setting (see, also, \citep[Lemma~4.3.1 and Theorem~4.3.1]{bensoussan:stochcontrol}).

\begin{lemma}\label{lem:rho1}
Under the same assumptions of Theorem~\ref{th:Zakai}, the process $\bigl(\rho_t(\one)\bigr)_{t \in [0,T]}$ satisfies for all $t \in [0,T]$
\begin{equation*}
\rho_t(\one) = 
\exp\left\{\int_0^t \gamma^{-1}(s) \pi_s(h_s) \, \dd \overline B_s - \dfrac 12 \int_0^t \norm{\gamma^{-1}(s) \pi_s(h_s)}^2 \, \dd s\right\}, \quad \bbP\text{-a.s.}
\end{equation*}
\end{lemma}

\begin{theorem}\label{th:KS}
Under the same assumptions of Theorem~\ref{th:Zakai}, the process $\bigl(\pi_t(\phi)\bigr)_{t \in [0,T]}$ satisfies for all $t \in [0,T]$ and all $\phi \in \dC^{1,2}_b([0,T] \times \R^m)$ the Kushner-Stratonovich equation
\begin{align}\label{eq:KS}
\pi_t(\phi_t) &= \pi_{0^-}(\phi_0) + \int_0^t \pi_s\bigl(\bigl[\partial_s + \cA_s\bigr] \phi_s\bigr) \, \dd s + \int_0^t \pi_{s^-}\bigl(\dD_x \phi_s\bigr) \, \dd \nu_s \notag
\\
&+ \int_0^t \gamma^{-1}(s) \Bigl\{\pi_s\bigl(\phi_s h_s\bigr) - \pi_s(\phi_s) \pi_s(h_s)\Bigr\} \, \bigl[\dd \overline B_s - \gamma^{-1}(s)\pi_s(h_s) \, \dd s\bigr] \notag
\\
&+\sum_{0 \leq s \leq t} \Bigl[\pi_{s^-}\bigl(\phi_s(\cdot + \Delta \nu_s) - \phi_s - \dD_x \phi_s \cdot \Delta \nu_s\bigr)\Bigr], \quad \bbP\text{-a.s.}
\end{align}
\end{theorem}

\begin{remark}
It is not difficult to show (see, e.g., \citep[Proposition~2.30]{bain:fundofstochfilt} or \citep[Theorem~4.3.4]{bensoussan:stochcontrol}), that
\begin{equation*}
I_t \coloneqq \overline B_t - \gamma^{-1}(t)\pi_t(h_t), \quad t \in [0,T],
\end{equation*}
is a $(\widetilde \bbP, \bbY)$-Brownian motion, the so-called \emph{innovation process}. This allows to rewrite the Kushner-Stratonovich equation in the (perhaps more familiar) form
\begin{align*}
\pi_t(\phi_t) &= \xi(\phi_0) + \int_0^t \pi_s\bigl(\bigl[\partial_s + \cA_s\bigr] \phi_s\bigr) \, \dd s + \int_0^t \pi_{s^-}\bigl(\dD_x \phi_s\bigr) \, \dd \nu_s \notag
\\
&+ \int_0^t \gamma^{-1}(s) \Bigl\{\pi_s\bigl(\phi_s h_s\bigr) - \pi_s(\phi_s) \pi_s(h_s)\Bigr\} \, \dd I_s \notag
\\
&+\sum_{0 \leq s \leq t} \Bigl[\pi_{s^-}\bigl(\phi_s(\cdot + \Delta \nu_s) - \phi_s - \dD_x \phi_s \cdot \Delta \nu_s\bigr)\Bigr], \quad \widetilde\bbP\text{-a.s.}, \quad t \in [0,T].
\end{align*}
Notice, however, that in this setting the innovation process is not a Brownian motion given \mbox{\emph{a priori}}, because it depends (through the density process $\eta$, and hence through $X$), on the initial law $\xi$ of the signal process and on process $\nu$. 
\end{remark}

\begin{remark}\label{rem:KSSPDE}
Similarly to what stated in Remark~\ref{rem:ZakaiSPDE}, if the jump times of the process $\nu$ do not accumulate over $[0,T]$, then the Kushner-Stratonovich equation can be split into successive nonlinear SPDEs between the jumps of $\nu$ (i.e., of $X$). Using the same notation of the aforementioned Remark, for any $\phi \in \dC^{1,2}_b([0,T] \times \R^m)$ and any $n \in \N_0$ we have $\bbP$-a.s.
\begin{equation}\label{eq:KSSPDE}
\left\{
\begin{aligned}
&\dd \pi_t(\phi_t) = \pi_t\bigl(\bigl[\partial_t + \cA_t\bigr] \phi_t\bigr) \, \dd t + \pi_{t^-}\bigl(\dD_x \phi_t\bigr) \, \dd \nu^c_t 
\\
&\quad + \gamma^{-1}(t) \Bigl\{\pi_t\bigl(\phi_t h_t\bigr) \!- \pi_t(\phi_t) \pi_t(h_t)\Bigr\} \bigl[\dd \overline B_t - \gamma^{-1}(t)\pi_t(h_t) \dd t\bigr], & &t \in [T_n \!\land\! T, T_{n+1} \!\land\! T),
\\
&\pi_{0^-}(\phi_0) = \xi(\phi_0),
\\
&\pi_{T_n}(\phi) = \pi_{{T_n}^-}\bigl(\phi_{T_n}(\cdot + \Delta \nu_{T_n})\bigr).
\end{aligned}
\right.
\end{equation}
\end{remark}


\section{Uniqueness of the solution to the Zakai equation}
\label{sec:uniqueness}

In this section we will address the issue of uniqueness of the solution to the Zakai equation \eqref{eq:Zakai}, under the requirement that the jump times of the process $\nu$ do not accumulate over $[0,T]$.
Proving uniqueness is essential to characterize completely the unnormalized filtering process $\rho$, defined in \eqref{eq:rho}, and is crucial in applications, e.g., in optimal control. Indeed, having ensured that \eqref{eq:Zakai} (or, equivalently, \eqref{eq:KS}) uniquely characterizes the conditional distribution of the signal given the observation, the filtering process can be employed as a state variable to solve the related separated optimal control problem (cf. \citep{bensoussan:stochcontrol}).

We follow the approach in \citep{kurtzxiong1999:SPDEs} (see, also, \citep[Chapter~7]{bain:fundofstochfilt} and \citep[Chapter~6]{xiong:intrtostochfiltth}).
The idea is to recast the measure-valued Zakai equation into an SPDE in the Hilbert space $H \coloneqq \dL^2(\R^m)$ and, therefore, to look for a density of $\rho$ in this space. To accomplish that, we will smooth solutions to \eqref{eq:Zakai} using the heat kernel, and we will then use estimates in $\dL^2(\R^m)$ in order to deduce the desired result. An important role in the subsequent analysis is played by the following lemma, whose proof can be found, e.g., in \citep[Solution to Exercise~7.2]{bain:fundofstochfilt}.

\begin{lemma}
\label{lem:density}
Let $\{\phi_k\}_{k \in \N}$ be an orthonormal basis of $H$ such that $\phi_k \in \dC_b(\R^m)$ for any $k \in \N$, and let $\mu \in \cM(\R^m)$ be a finite measure. If
\begin{equation*}
\sum_{k \in \N} [\mu(\phi_k)]^2 < +\infty,
\end{equation*}
then $\mu$ is absolutely continuous with respect to Lebesgue measure on $\R^m$ and its density is square-integrable.
\end{lemma}

Let $\psi_\epsilon$ be the heat kernel, i.e., the function defined for each $\epsilon > 0$ as
\begin{equation*}
\psi_\epsilon(x) \coloneqq \frac{1}{(2\pi\epsilon)^{m/2}} \de^{-\frac{\norm{x}^2}{2\epsilon}}, \quad x \in \R^m,
\end{equation*}
and for any Borel-measurable and bounded $f$ and $\epsilon > 0$ define the operator
\begin{equation*}
T_\epsilon f(x) \coloneqq \int_{\R^m} \psi_\epsilon(x-y) \, f(y) \, \dd y, \qquad x \in \R^m.
\end{equation*}
We also define the operator $T_\epsilon \colon \cM(\R^m) \to \cM(R^m)$ given by
\begin{equation*}
T_\epsilon \mu(f) \coloneqq \mu(T_\epsilon f) = \int_{\R^m} f(y) \underbrace{\int_{\R^m} \psi_\epsilon(x-y) \, \mu(\dd x)}_{\coloneqq T_\epsilon \mu(y)}  \dd y = \int_{\R^m} f(y) \, T_\epsilon \mu(y) \, \dd y.
\end{equation*}
The equalities above imply that for any $\mu \in \cM(\R^m)$ the measure $T_\epsilon \mu$ always possesses a density with respect to Lebesgue measure, that we will still denote by $T_\epsilon \mu$. 

\begin{remark}\label{rem:Tepsreg}
It is important to notice that, by \citep[Exercise~7.3, point~ii.]{bain:fundofstochfilt}, $T_\epsilon \mu \in W^2_k(\R^m)$, for any $\mu \in \cM(\R^m)$, $\epsilon > 0$, and $k \in \N$.
\end{remark}

Further properties of these operators that will be used in the sequel are listed in the following Lemma (for its proof see, e.g., \citep[Solution to Exercise~7.3]{bain:fundofstochfilt} and \citep[Lemma~6.7, Lemma~6.8]{xiong:intrtostochfiltth}).
\begin{lemma}\label{lem:Tprop}
For any $\mu \in \cM(\R^m)$, $h \in H$, and $\epsilon > 0$ we have that:
\begin{enumerate}[label=\roman*.]
\item $\norm{T_{2\epsilon} |\mu|}_H \leq \norm{T_{\epsilon} |\mu|}_H$, where $|\mu|$ denotes the total variation measure of $\mu$;
\item $\norm{T_\epsilon h}_H \leq \norm{h}_H$;
\item $\langle T_\epsilon \mu, h \rangle = \mu(T_\epsilon h)$;
\item If, in addition, $\partial_i h \in H$, $i = 1, \dots, m$, then $\partial_i T_\epsilon h = T_\epsilon \partial_i h$ (with the partial derivative understood in the weak sense).
\item If $\phi \in \dC^1_b(\R^m)$, then $\partial_i T_\epsilon \phi = T_\epsilon(\partial_i \phi)$.
\end{enumerate}
\end{lemma}

In this section we will work under the following hypotheses, in addition to Assumptions~\ref{hyp:main} and \ref{hyp:h}, concerning coefficients $b$, $\sigma$ and $h$ appearing in SDEs~\eqref{eq:XSDE} and \eqref{eq:YSDEctrl}. In what follows we will use the shorter notation 
\begin{equation}\label{eq:adef}
a(t,x) \coloneqq \frac 12 \sigma\sigma^*(t,x), \quad t \in [0,T], \, x \in \R^m.
\end{equation}
\begin{assumption}\label{hyp:bddcoeff}
There exist constants $K_b$, $K_\sigma$, $K_h$, such that, for all $i,j = 1, \dots,m$, all $\ell = 1, \dots, n$, all $t \in [0,T]$, and all $x \in \R^m$,
\begin{equation*}
\abs{b_i(t,x)} \leq K_b, \quad \abs{a_{ij}(t,x)} \leq K_\sigma, \quad \abs{h_\ell(t,x)} \leq K_h.
\end{equation*}
\end{assumption}

In the next section, we obtain the uniqueness result for the solution to the Zakai equation when the process $\nu$ has continuous paths. This will be then exploited in Section~\ref{sec:nudisc} in order to obtain the uniqueness claim when $\nu$ has jump times that do not accumulate over $[0,T]$.


\subsection{\texorpdfstring{The case in which $\nu$ has continuous paths}{The case in which nu has continuous paths}}
\label{sec:nucont}

We start our analysis with the following Lemma, which will play a fundamental role in the sequel. Its proof can be found in Appendix \ref{app:technical}.
\begin{lemma}
\label{lem:tousegronwall}
Suppose that Assumption \ref{hyp:bddcoeff} holds. Let $\zeta = (\zeta_t)_{t \in [0,T]}$ be a $\bbY$-adapted, c\`adl\`ag, $\cM_+(\R^m)$-valued solution of \eqref{eq:Zakai}, with $\zeta_{0^-} = \xi \in \cP(\R^m)$. If $\nu$ is continuous, then for any $\epsilon > 0$
\begin{equation*}
\bbE[\sup_{t \in [0,T]} \norm{T_\epsilon \zeta_{t^-}}_H^2] < +\infty.
\end{equation*}
\end{lemma}

The next result is a useful estimate.
\begin{proposition}\label{prop:znormest}
Suppose that Assumption \ref{hyp:bddcoeff} holds. Let $\zeta = (\zeta_t)_{t \in [0,T]}$ be a $\bbY$-adapted, c\`adl\`ag, $\cM(\R^m)$-valued solution of \eqref{eq:Zakai}, with $\zeta_{0^-} = \xi \in \cP(\R^m)$. Define the process
\begin{equation}\label{eq:Adef}
A_t \coloneqq t + \sum_{i=1}^m \abs{\nu^i}_t, \quad t \in [0,T].
\end{equation}
If $\nu$ is continuous and if, for any $\epsilon > 0$, $\bbE[\sup_{t \in [0,T]} \norm{T_\epsilon \abs{\zeta}_{t^-}}_H^2] < +\infty$, then there exists a constant $M>0$ such that, for each $\epsilon > 0$ and all $\bbF$-stopping times $\tau \leq t$, $t \in [0,T]$,
\begin{equation}\label{eq:tepszetaest}
\bbE[\norm{T_\epsilon \zeta_{\tau^-}}^2_H] \leq \norm{T_\epsilon \zeta_{0^-}}_H^2 + M \int_0^{\tau^-} \bbE[\norm{T_\epsilon \abs{\zeta}_{s^-}}^2_H] \, \dd A_s.
\end{equation}
\end{proposition}

\begin{proof}
To ease notations, for any $\epsilon > 0$ denote by $Z^\epsilon$ the process $Z^\epsilon_t \coloneqq T_\epsilon \zeta_t$, $t \geq 0$.
Fix $\epsilon > 0$ and consider an orthonormal basis $\{\phi_k\}_{k \in \N}$ of $H$ such that $\phi_k \in \dC_b^2(\R^m)$, for any $k \in \N$. Writing the Zakai equation for the function $T_\epsilon \phi_k$ (recall that $\nu$ is continuous by assumption) we get:
\begin{equation}\label{eq:Zakaimoll}
\zeta_t(T_\epsilon \phi_k) 
= \xi(T_\epsilon \phi_k) + \int_0^t \zeta_s\bigl(\cA_s T_\epsilon \phi_k\bigr) \, \dd s 
+ \int_0^t \zeta_{s^-}\bigl(\dD_x T_\epsilon \phi_k\bigr) \, \dd \nu_s + \int_0^t \gamma^{-1}(s) \zeta_s(T_\epsilon \phi_k h_s) \, \dd \overline B_s,
\end{equation}
for all $t \in [0,T]$.
Notice that, for any $\phi \in \dC^2_b(\R^m)$ and any $t \in [0,T]$, we can write:
\begin{equation*}
\cA_t \phi(x) = \sum_{i=1}^m b_i(t,x) \partial_i \phi(x) + \sum_{i,j=1}^m a_{ij}(t,x) \partial_{ij} \phi(x), \quad x \in \R^m, 
\end{equation*}
where $a$ is the function defined in \eqref{eq:adef}.
For any $i,j = 1, \dots, m$, $\ell = 1, \dots, n$, and $t \in [0,T]$, we define the random measures on $\R^m$:
\begin{equation*}
b^i_t\zeta_t(\dd x) \coloneqq b_i(t,x) \zeta_t(\dd x),
\quad
a^{ij}_t\zeta_t(\dd x) \coloneqq a_{ij}(t,x) \zeta_t(\dd x),
\quad
\gamma h^\ell_t\zeta_t(\dd x) \coloneqq \sum_{p=1}^n \gamma^{-1}_{\ell p}(t) h_p(t,x) \zeta_t(\dd x).
\end{equation*}
These measures are $\bbP$-almost surely finite, for any $t \in [0,T]$, thanks to Assumption \ref{hyp:bddcoeff} and to \eqref{eq:invgammanorm} (see also \eqref{eq:gammazetahest} for the last measure).

Applying Lemma~\ref{lem:Tprop} and the integration by parts formula we get:
\begin{align*}
&\mathop{\phantom{=}}\zeta_t\bigl(\cA_t T_\epsilon \phi_k\bigr)
= \sum_{i=1}^m \int_{\R^m} b_i(t,x) \partial_i T_\epsilon \phi_k(x) \, \zeta_t(\dd x) + \sum_{i,j=1}^m \int_{\R^m} a_{ij}(t,x) \partial_{ij} T_\epsilon \phi_k(x) \, \zeta_t(\dd x) \\
&= \sum_{i=1}^m \int_{\R^m} b_i(t,x) T_\epsilon \partial_i \phi_k(x) \, \zeta_t(\dd x) + \sum_{i,j=1}^m \int_{\R^m} a_{ij}(t,x) T_\epsilon \partial_{ij} \phi_k(x) \, \zeta_t(\dd x) \\
&= \sum_{i=1}^m  b^i_t\zeta_t(T_\epsilon \partial_i \phi_k) + \sum_{i,j=1}^m  a^{ij}_t\zeta_t(T_\epsilon \partial_{ij} \phi_k)
\\ 
&= \sum_{i=1}^m \inprod{T_\epsilon (b^i_t\zeta_t)}{\partial_i \phi_k} + \sum_{i,j=1}^m \inprod{T_\epsilon (a^{ij}_t\zeta_t)}{\partial_{ij} \phi_k} 
= \sum_{i,j=1}^m \langle \phi_k, \partial_{ij} T_\epsilon(a^{ij}_t\zeta_t) \rangle - \sum_{i=1}^m \langle \phi_k, \partial_i T_\epsilon(b^i_t\zeta_t)\rangle.
\end{align*}
In a similar way, we obtain $\zeta_t\bigl(\partial_i T_\epsilon \phi_k\bigr) = - \langle \phi_k, \partial_i T_\epsilon \zeta_t \rangle$, and
\begin{equation*}
\sum_{j=1}^n \gamma^{-1}_{ij}(t) \zeta_t\bigl(T_\epsilon \phi_k h^j_t\bigr) = \langle \phi_k, T_\epsilon(\gamma h^i_t \zeta_t) \rangle, \quad i = 1, \dots, n.
\end{equation*} 

Putting together all these facts, we can rewrite \eqref{eq:Zakaimoll} as
\begin{align*}
\langle \phi_k, Z_t^\epsilon \rangle 
&= \langle \phi_k, Z_{0^-}^\epsilon\rangle + \sum_{i,j=1}^m \int_0^t \langle \phi_k, \partial_{ij} T_\epsilon(a^{ij}_s\zeta_s) \rangle \, \dd s - \sum_{i=1}^m \int_0^t \langle \phi_k, \partial_i T_\epsilon(b^i_s\zeta_s)\rangle \, \dd s 
\\
&- \sum_{i=1}^m \int_0^t \langle \phi_k, \partial_i T_\epsilon \zeta_{s^-} \rangle \, \dd \nu^i_s 
+ \sum_{i= 1}^n \int_0^t \langle \phi_k, T_\epsilon(\gamma h^i_s \zeta_s) \rangle \, \dd \overline B_s^i, \quad \bbP\text{-a.s.}, \quad t \in [0,T].
\end{align*}

Applying It\^o's formula we get that, for all $t \in [0,T]$, $\bbP$-a.s.,
\begin{align*}
\langle \phi_k, Z_t^\epsilon \rangle^2 
&= \langle \phi_k, Z_{0^-}^\epsilon\rangle^2 + \sum_{i,j=1}^m \int_0^t 2\langle \phi_k, Z_s^\epsilon \rangle \, \langle \phi_k, \partial_{ij} T_\epsilon(a^{ij}_s\zeta_s) \rangle \, \dd s
\\
&- \sum_{i=1}^m \int_0^t 2 \langle \phi_k, Z_s^\epsilon \rangle \langle \phi_k, \partial_i T_\epsilon(b^i_s\zeta_s)\rangle \, \dd s 
+ \sum_{i=1}^n \int_0^t \langle \phi_k, T_\epsilon(\gamma h^i_s \zeta_s) \rangle^2 \, \dd s
\\
&- \sum_{i=1}^m \int_0^t 2 \langle \phi_k, Z_{s^-}^\epsilon \rangle \langle \phi_k, \partial_i T_\epsilon \zeta_{s^-} \rangle \, \dd \nu^i_s 
+ \sum_{i=1}^n  \int_0^t 2 \langle \phi_k, Z_s^\epsilon \rangle \langle \phi_k, T_\epsilon(\gamma h^i_s \zeta_s) \rangle \, \dd \overline B_s^i.
\end{align*}

Using Assumption \ref{hyp:bddcoeff} and \eqref{eq:invgammanorm}, it is possible to show that the stochastic integral with respect to Brownian motion $\overline B$ is a $\bbP$-martingale. By the optional sampling theorem, this stochastic integral has zero expectation even when evaluated at any bounded stopping time. Therefore, picking an $\bbF$-stopping time $\tau \leq t$, for arbitrary $t \in [0,T]$, summing over $k$ up to $N \in \N$, and taking the expectation, by Fatou's lemma we have that
\begin{align}\label{eq:znormest}
&\mathop{\phantom{=}}\bbE\left[\norm{Z_{\tau^-}^\epsilon}_H^2 \right] 
= \bbE\biggl[\lim_{N \to \infty} \sum_{k=1}^N \langle \phi_k, Z_{\tau^-}^\epsilon \rangle^2 \biggr]
\leq \liminf_{N \to \infty} \bbE\biggl[\sum_{k=1}^N \langle \phi_k, Z_{\tau^-}^\epsilon \rangle^2 \biggr]
\leq \norm{Z_{0^-}^\epsilon}_H^2 \notag
\\
&+ \liminf_{N \to \infty} \Biggl\{\sum_{i,j=1}^m \bbE\biggl[\int_0^{\tau^-} \sum_{k=1}^N  2\langle \phi_k, Z_s^\epsilon \rangle \, \langle \phi_k, \partial_{ij} T_\epsilon(a^{ij}_s\zeta_s) \rangle \, \dd s\biggr] \notag 
\\
&\quad- \sum_{i=1}^m \bbE\biggl[\int_0^{\tau^-} \sum_{k=1}^N 2 \langle \phi_k, Z_s^\epsilon \rangle \langle \phi_k, \partial_i T_\epsilon(b^i_s\zeta_s)\rangle \, \dd s\biggr]
+ \sum_{i=1}^n \bbE\biggl[\int_0^{\tau^-} \sum_{k=1}^N \langle \phi_k, T_\epsilon(\gamma h^i_s \zeta_s) \rangle^2 \, \dd s\biggr] \notag
\\
&\quad- \sum_{i=1}^m \bbE\biggl[\int_0^{\tau^-} \sum_{k=1}^N 2 \langle \phi_k, Z_{s^-}^\epsilon \rangle \langle \phi_k, \partial_i T_\epsilon \zeta_{s^-} \rangle \, \dd \nu^i_s\biggr]\Biggr\},
\end{align}
where we used the fact that, since $Z_{0^-}^\epsilon \in H$, $\lim\limits_{N \to \infty} \sum\limits_{k=1}^N \langle \phi_k, Z_{0^-}^\epsilon \rangle^2 = \norm{Z_{0^-}^\epsilon}_H^2$. More generally, since $Z_t^\epsilon \in H$, for all $t \in [0,T]$, $\bbP$-a.s. (cf. Remark \ref{rem:Tepsreg}), we have that
\begin{equation}\label{eq:ztnormest}
\sum_{k=1}^N  \langle \phi_k, Z_t^\epsilon \rangle^2 \leq \sum_{k=1}^\infty  \langle \phi_k, Z_t^\epsilon \rangle^2 = \norm{Z_t^\epsilon}_H^2, \quad t \in [0,T].
\end{equation}

We want now to estimate the quantities appearing inside the limit inferior, in order to exchange the limit and the integrals in \eqref{eq:znormest}.
First of all, let us notice that, thanks to Assumption \ref{hyp:bddcoeff}, the following estimates hold $\bbP$-a.s., for all $i,j = 1, \dots, m$, all $\ell = 1, \dots, n$, and all $t \in [0,T]$:
\begin{align*}
&\norm{\partial_{ij} T_\epsilon(a^{ij}_t\zeta_t)}_H^2 \leq K_1 \norm{T_\epsilon \abs{\zeta}_t}_H^2,
&
&\norm{\partial_i T_\epsilon(b^i_t\zeta_t)}_H^2 \leq K_2 \norm{T_\epsilon \abs{\zeta}_t}_H^2,
\\
&\norm{T_\epsilon(\gamma h^\ell_t \zeta_t)}_H^2 \leq K_3 \norm{T_\epsilon \abs{\zeta}_t}_H^2,
&
&\norm{\partial_i T_\epsilon \zeta_t}_H^2 \leq K_4 \norm{T_\epsilon \abs{\zeta}_t}_H^2,
\end{align*}
where $K_1 = K_1(\epsilon, m, \sigma)$, $K_2 = K_2(\epsilon, m, b)$, $K_3 = K_3(n, h, \gamma)$, $K_4 = K_4(\epsilon, m)$.
They can be proved following a reasoning analogous to that of \citep[Lemma~7.5]{bain:fundofstochfilt} (see also \citep[Chapter~6]{xiong:intrtostochfiltth}).

Recalling that $2|ab| \leq a^2 + b^2$, for all $a,b \in \R$, using the estimates provided above, Lemma~\ref{lem:Tprop}, and \eqref{eq:ztnormest}, we get that, for all $N \in \N$, all $i,j = 1, \dots, m$, and all $s \in [0,T]$,
\begin{align*}
&\mathop{\phantom{\leq}} \ind_{s < \tau} \sum_{k=1}^N  2\langle \phi_k, Z_s^\epsilon \rangle \, \langle \phi_k, \partial_{ij} T_\epsilon(a^{ij}_s\zeta_s) \rangle
\leq \sum_{k=1}^N  \langle \phi_k, Z_s^\epsilon \rangle^2 + \sum_{k=1}^N \langle \phi_k, \partial_{ij} T_\epsilon(a^{ij}_s\zeta_s) \rangle^2
\\
&\leq \norm{Z_s^\epsilon}_H^2 + \norm{\partial_{ij} T_\epsilon(a^{ij}_s\zeta_s)}_H^2
\leq (1+K_1) \norm{T_{\epsilon/2} \abs{\zeta}_s}_H^2.
\end{align*}
With analogous computations, we get, for all $i = 1, \dots, m$, all $N \in \N$, and all $s \in [0,T]$,
\begin{gather*}
\ind_{s < \tau} \sum_{k=1}^N 2 \langle \phi_k, Z_s^\epsilon \rangle \langle \phi_k, \partial_i T_\epsilon(b^i_s\zeta_s)\rangle
\leq (1+K_2) \norm{T_{\epsilon/2} \abs{\zeta}_s}_H^2,
\\
\ind_{s < \tau} \sum_{k=1}^N 2 \langle \phi_k, Z_{s^-}^\epsilon \rangle \langle \phi_k, \partial_i T_\epsilon \zeta_{s^-} \rangle
\leq (1+K_4) \norm{T_{\epsilon/2} \abs{\zeta}_s}_H^2,
\end{gather*}
and, for all $N \in \N$ and all $s \in [0,T]$,
\begin{equation*}
\sum_{i=1}^n \ind_{s < \tau} \sum_{k=1}^N \langle \phi_k, T_\epsilon(\gamma h^i_s \zeta_s) \rangle^2
\leq n K_3 \norm{T_\epsilon \abs{\zeta}_s}_H^2.
\end{equation*}
The terms appearing on the r.h.s. of these estimates are $\dd t \otimes \dd \bbP$- and \mbox{$\dd \abs{\nu^i}_t \otimes \dd \bbP$-integrable} on $[0,T] \times \Omega$, for all $i = 1,\dots,m$, since, for any $\epsilon > 0$,
\begin{align*}
&\bbE\left[\int_0^T \norm{T_\epsilon \abs{\zeta}_s}_H^2 \, \dd s \right] \leq T \bbE[\sup_{s \in [0,T]} \norm{T_\epsilon \abs{\zeta}_s}_H^2] < +\infty,
\\
&\bbE\left[\int_0^T \norm{T_\epsilon \abs{\zeta}_s}_H^2 \, \dd \abs{\nu^i}_s \right] \leq K \bbE[\sup_{s \in [0,T]} \norm{T_\epsilon \abs{\zeta}_s}_H^2] < +\infty.
\end{align*}
Therefore, by the dominated convergence theorem, we can pass to the limit in \eqref{eq:znormest}, as $N \to \infty$,
\begin{align}\label{eq:znormest2}
&\mathop{\phantom{=}}\bbE\left[\norm{Z_{\tau^-}^\epsilon}_H^2 \right] 
\leq \norm{Z_{0^-}^\epsilon}_H^2
+ \sum_{i,j=1}^m \bbE\biggl[\int_0^{\tau^-} \! 2\langle Z_s^\epsilon, \partial_{ij} T_\epsilon(a^{ij}_s\zeta_s) \rangle \, \dd s\biggr] - \sum_{i=1}^m \bbE\biggl[\int_0^{\tau^-} \! 2 \langle Z_s^\epsilon, \partial_i T_\epsilon(b^i_s\zeta_s)\rangle \, \dd s\biggr] \notag
\\
&+ \sum_{i=1}^n \bbE\biggl[\int_0^{\tau^-} \norm{T_\epsilon(\gamma h^i_s \zeta_s)}_H^2 \, \dd s\biggr] - \sum_{i=1}^m \bbE\biggl[\int_0^{\tau^-} \langle Z_{s^-}^\epsilon, \partial_i T_\epsilon \zeta_{s^-} \rangle \, \dd \nu^i_s\biggr],
\end{align}

We finally get the claim, bounding the terms on the r.h.s.\ of \eqref{eq:znormest2} by using the following results: for the second one, apply \citep[Lemma~6.11]{xiong:intrtostochfiltth}; for the third and the last one, apply \citep[Lemma~6.10]{xiong:intrtostochfiltth}; for the fourth one, use the fact that the constant $K_3$ above does not depend on $\epsilon$.
\end{proof}

Proposition~\ref{prop:znormest} allows to deduce that any $\cM_+(\R^m)$-valued solution of the Zakai equation \eqref{eq:Zakai} admits a density with respect to Lebesgue measure.

\begin{proposition}\label{prop:densitynucont}
Suppose that Assumption \ref{hyp:bddcoeff} holds. Let $\zeta = (\zeta_t)_{t \in [0,T]}$ be a $\bbY$-adapted, c\`adl\`ag, $\cM_+(\R^m)$-valued solution of \eqref{eq:Zakai}, with $\zeta_{0^-} = \xi \in \cP(\R^m)$.
If $\nu$ is continuous and if $\xi$ admits a square-integrable density with respect to Lebesgue measure on $\R^m$, then there exists an $H$-valued process $Z = (Z_t)_{t \in [0,T]}$ such that, for all $t \in [0,T]$,
\begin{equation*}
\zeta_t(\dd x) = Z_t(x) \dd x, \quad \bbP\text{-a.s.}
\end{equation*}
Moreover, $Z$ is $\bbY$-adapted, continuous, and satisfies $\bbE[\norm{Z_t}_H^2] < +\infty$, for all $t \in [0,T]$.
\end{proposition}

\begin{proof}
As a consequence of Lemma~\ref{lem:tousegronwall}, the assumptions of Proposition~\ref{prop:znormest} hold and we have that
for each $\epsilon > 0$ and all $\bbF$-stopping times $\tau \leq t$, $t \in [0,T]$,
\begin{equation*}
\bbE[\norm{T_\epsilon \zeta_{\tau^-}}^2_H] \leq \norm{T_\epsilon \zeta_{0^-}}_H^2 + M \int_0^{\tau^-} \bbE[\norm{T_\epsilon \zeta_{s^-}}^2_H] \, \dd A_s.
\end{equation*}
Therefore, we can apply Lemma~\ref{lem:gronwall} and get that, for all $t \in [0,T]$,
\begin{equation}\label{eq:rhonormest}
\bbE[\norm{T_\epsilon \zeta_{t^-}}^2_H] = \bbE[\norm{T_\epsilon \zeta_t}^2_H] \leq \norm{T_\epsilon \zeta_{0^-}}_H^2 \de^{M(T+mK)},
\end{equation}
where we used the fact that $\zeta$ is continuous, since $\nu$ is, and that $A_t \leq A_T \leq T+mK$, for all $t \in [0,T]$.
Notice that, denoting by $Z_{0^-}$ the density of $\xi$ with respect to Lebesgue measure on $\R^m$,
\begin{equation*}
T_\epsilon \zeta_{0^-}(y) = \int_{\R^m} \psi_{\epsilon}(x-y) \, \xi(\dd x) = \int_{\R^m} \psi_{\epsilon}(x-y) \, Z_{0^-}(x) \, \dd x = T_\epsilon Z_{0^-}(y), \quad y \in \R^m.
\end{equation*}
By point~ii. of Lemma~\ref{lem:Tprop} and since the constants appearing in \eqref{eq:rhonormest} do not depend on $\epsilon$, we get
\begin{equation*}
\sup_{\epsilon > 0} \bbE[\norm{T_\epsilon \zeta_t}^2_H] \leq \norm{Z_{0^-}}_H^2 \de^{M(T+mK)}, \quad t \in [0,T].
\end{equation*}
Taking, as in the Proof of Proposition~\ref{prop:znormest}, an orthonormal basis $\{\phi_k\}_{k \in \N}$ of $H$ such that $\phi_k \in \dC_b^2(\R^m)$, for any $k \in \N$, the dominated convergence theorem entails that, for all $k \in \N$,
\begin{equation*}
\lim_{\epsilon \to 0} \langle T_\epsilon \zeta_t, \phi_k \rangle = \lim_{\epsilon \to 0} \int_{\R^m} \left\{ \int_{\R^m} \psi_\epsilon(x-y) \phi_k(y) \, \dd y \right\} \zeta_t(\dd x) = \int_{\R^m} \phi_k(x) \, \zeta_t(\dd x) = \zeta_t(\phi_k).
\end{equation*}
Applying Fatou's Lemma we get that, for all $t \in [0,T]$,
\begin{align}\label{eq:ptnormest}
&\mathop{\phantom{\leq}}\bbE\left[\sum_{k=1}^\infty \zeta_t(\phi_k)^2 \right] = \bbE\left[\sum_{k=1}^\infty \lim_{\epsilon \to 0} \langle T_\epsilon \zeta_t, \phi_k \rangle^2 \right] 
\leq \liminf_{\epsilon \to 0} \bbE\left[\sum_{k=1}^\infty \langle T_\epsilon \zeta_t, \phi_k \rangle^2 \right] \notag
\\
&\leq \sup_{\epsilon > 0} \bbE[\norm{T_\epsilon \zeta_t}_H^2] \leq \norm{Z_{0^-}}_H^2 \de^{M(T+mK)} < +\infty,
\end{align}
and hence, from Lemma~\ref{lem:density} we deduce that, $\bbP$-a.s., $\zeta_t$ is absolutely continuous with respect to Lebesgue measure on $\R^m$, for all $t \in [0,T]$. Moreover, its density process $Z = (Z_t)_{t \in [0,T]}$ takes values in $H$ and, by standard results, is $\bbY$-adapted and continuous (because $\nu$ is).

Finally, since $\zeta_t(\phi_k) = \int_{\R^m} \phi_k(x) Z_t(x) \, \dd x = \langle \phi_k, Z_t \rangle$, for all $k \in \N$, and all $t \in [0,T]$, we get
\begin{equation*}
\bbE[\norm{Z_t}_H^2] = \bbE\left[\sum_{k=1}^\infty \langle \phi_k, Z_t \rangle^2 \right] = \bbE\left[\sum_{k=1}^\infty \zeta_t(\phi_k)^2 \right] < +\infty, \quad t \in [0,T]. \qedhere
\end{equation*}
\end{proof}

We are now ready to state our first uniqueness result for the solution to the Zakai equation, in the case where $\nu$ is continuous.
\begin{theorem}\label{th:Zakaiuniqnucont}
Suppose that Assumptions~\ref{hyp:main}, \ref{hyp:h}, \ref{hyp:bddcoeff}, and \eqref{eq:xi3} hold. If $\nu$ is continuous and if $\xi \in \cP(\R^m)$ admits a square-integrable density with respect to Lebesgue measure on $\R^m$, then the unnormalized filtering process $\rho$, defined in \eqref{eq:rho}, is the unique $\bbY$-adapted, continuous, $\cM_+(\R^m)$-valued solution to the Zakai equation \eqref{eq:Zakai}. 

Moreover, there exists a $\bbY$-adapted, continuous, $H$-valued process $p = (p_t)_{t \in [0,T]}$ satisfying, for all $t \in [0,T]$, $\bbE[\norm{p_t}_H^2] < +\infty$ and $\rho_t(\dd x) = p_t(x) \dd x$, $\bbP$-a.s.
\end{theorem}

\begin{proof}
Clearly, the unnormalized filtering process $\rho$, defined in \eqref{eq:rho}, is a $\bbY$-adapted, continuous (since $\nu$ is), $\cM_+(\R^m)$-valued solution to \eqref{eq:Zakai}. Therefore, the second part of the statement follows directly from Proposition~\ref{prop:densitynucont}.

Uniqueness can be established as follows. Let $\zeta^{(1)}, \zeta^{(2)}$ be two $\bbY$-adapted, c\`adl\`ag, $\cM_+(\R^m)$-valued solutions to \eqref{eq:Zakai}. Define $\zeta \coloneqq \zeta^{(1)}-\zeta^{(2)} \in \cM(\R^m)$ and let $Z \coloneqq Z^{(1)} - Z^{(2)} \in H$ be its density process, where $Z^{(1)}$ and $Z^{(2)}$ are the density processes of $\zeta^{(1)}$ and $\zeta^{(2)}$, respectively, which exist thanks to Proposition~\ref{prop:densitynucont}.

Standard facts from measure theory show that, for all non-negative, bounded, measurable functions $\phi \colon \R^m \to \R$ and all $t \in [0,T]$, $\abs{\zeta}_t(\phi) \leq \zeta^{(1)}_t(\phi) + \zeta^{(2)}_t(\phi)$. From this fact, applying Lemma~\ref{lem:tousegronwall} we deduce that
\begin{equation*}
\bbE[\sup_{t \in [0,T]} \norm{T_\epsilon \abs{\zeta}_{t^-}}_H^2] \leq 2 \bbE[\sup_{t \in [0,T]} \norm{T_\epsilon \zeta^{(1)}_{t^-}}_H^2] + 2\bbE[\sup_{t \in [0,T]} \norm{T_\epsilon \zeta^{(2)}_{t^-}}_H^2] < +\infty.
\end{equation*}
Therefore, from Proposition~\ref{prop:znormest} we get that for all $\epsilon > 0$ and all $\bbF$-stopping times $\tau \leq t$, $t \in [0,T]$,
\begin{equation*}
\bbE[\norm{T_\epsilon \zeta_{\tau^-}}^2_H] \leq M \int_0^{\tau^-} \bbE[\norm{T_\epsilon \abs{\zeta}_{s^-}}^2_H] \, \dd A_s,
\end{equation*}
where $A$ is defined in \eqref{eq:Adef}.
An application of the dominated convergence theorem shows that $\norm{T_\epsilon \abs{\zeta}_{t^-}}_H^2 \longrightarrow \norm{Z_t}^2_H$, as $\epsilon \to 0$, for all $t \in [0,T]$, and hence, by Fatou's lemma
\begin{align*}
&\mathop{\phantom{\leq}}\bbE[\norm{Z_t}^2_H] = \bbE[\lim_{\epsilon \to 0} \norm{T_\epsilon \zeta_{\tau^-}}^2_H] \leq \liminf_{\epsilon \to 0} \bbE[\norm{T_\epsilon \zeta_{\tau^-}}^2_H] 
\\
&\leq \liminf_{\epsilon \to 0} M \int_0^{\tau^-} \bbE[\norm{T_\epsilon \abs{\zeta}_{s^-}}^2_H] \, \dd A_s = M \int_0^{\tau^-} \bbE[\norm{Z_{s^-}}^2_H] \, \dd A_s.
\end{align*}
Finally, Proposition~\ref{prop:densitynucont} ensures that 
\begin{equation*}
\bbE[\norm{Z_{t^-}}^2_H] \leq 2 \bbE[\norm{Z^{(1)}_{t^-}}^2_H] + 2 \bbE[\norm{Z^{(2)}_{t^-}}^2_H] < +\infty, \quad \text{for all } t \in [0,T],
\end{equation*}
This allows us to use Lemma~\ref{lem:gronwall} to get that, for all $t \in [0,T]$, $\bbE[\norm{Z_{t^-}}^2_H] = \bbE[\norm{Z_t}^2_H] = 0$,
whence we obtain $\norm{Z_t}^2_H = 0$, $\bbP$-a.s., and therefore uniqueness of the solution to the Zakai equation.
\end{proof}


\subsection{\texorpdfstring{The case in which the jump times of $\nu$ do not accumulate}{The case in which the jump times of nu do not accumulate}}
\label{sec:nudisc}

Exploiting the recursive structure of \eqref{eq:ZakaiSPDE}, we can prove uniqueness of the solution to the Zakai equation \eqref{eq:Zakai}, also in the case where the jump times of $\nu$ do not accumulate.

\begin{theorem}\label{th:Zakaiuniqnujump}
Suppose that Assumptions~\ref{hyp:main}, \ref{hyp:h}, \ref{hyp:bddcoeff}, and \eqref{eq:xi3} hold.
If the jump times of $\nu$ do not accumulate over $[0,T]$ and if $\xi \in \cP(\R^m)$ admits a square-integrable density with respect to Lebesgue measure on $\R^m$, then the unnormalized filtering process $\rho$, defined in \eqref{eq:rho}, is the unique $\bbY$-adapted, c\`adl\`ag, $\cM_+(\R^m)$-valued solution to the Zakai equation \eqref{eq:Zakai}. 

Moreover, there exists a $\bbY$-adapted, c\`adl\`ag, $H$-valued process $p = (p_t)_{t \in [0,T]}$ satisfying, for all $t \in [0,T]$, $\bbE[\norm{p_t}_H^2] < +\infty$ and $\rho_t(\dd x) = p_t(x) \dd x$, $\bbP$-a.s.
\end{theorem}

\begin{proof}
Let us denote by $\rho$ the unnormalized filtering process associated with the initial law $\xi$ and process $\nu$, and by $p_{0^-}$ the density of $\xi$ with respect to Lebesgue measure on $\R^m$. Let $T_0 = 0$ and define the sequence of jump times of $\nu$
\begin{equation*}
T_n \coloneqq \inf\{t > T_{n-1} \colon \Delta \nu_t \neq 0\}, \quad n \in \N,
\end{equation*}
with the usual convention $\inf \emptyset = +\infty$.
Recall that also $T_0$ can be a jump time of $\nu$. Moreover, since the jump times of $\nu$ do not accumulate over $[0,T]$, we have that $T_n \leq T_{n+1}$, $\bbP$-a.s., and $T_n < +\infty \, \Longrightarrow T_n < T_{n+1}$, for all $n \in \N_0$.

We start noticing that the formula $\rho_{T_n}(\phi) = \rho_{{T_n}^-}\bigl(\phi_{T_n}(\cdot + \Delta \nu_{T_n})\bigr)$, $n \in \N_0$, appearing in \eqref{eq:ZakaiSPDE} holds for all $\phi \in \dC_b(\R^m)$. Indeed, continuity of the observation filtration $\bbY$ implies that
\begin{equation*}
\widetilde{\bbE}[\phi(X_{T_n^-}) \mid \cY_{T_n}] = \widetilde{\bbE}[\phi(X_{T_n^-}) \mid \cY_{T_n^-}] = \pi_{T_n^-}(\phi).
\end{equation*}
Using continuity of process $\eta$, Kallianpur-Striebel formula \eqref{eq:KallianpurStriebel}, and the freezing lemma, we get
\begin{align*}
&\mathop{\phantom{=}} \rho_{T_n}(\phi) = \widetilde{\bbE}[\phi(X_{T_n}) \mid \cY_{T_n}] \, \bbE\bigl[\eta_{T_n} \bigm| \cY \bigr] = \widetilde{\bbE}[\phi(X_{T_n^-} + \Delta \nu_{T_n}) \mid \cY_{T_n}] \, \bbE\bigl[\eta_{T_n^-} \bigm| \cY \bigr] \\
& = \pi_{T_n^-}(\phi(\cdot + \Delta \nu_{T_n})) \, \bbE\bigl[\eta_{T_n^-} \bigm| \cY \bigr] = \rho_{T_n^-}(\phi(\cdot + \Delta \nu_{T_n})),
\end{align*}
for all $\phi \in \dC_b(\R^m)$ and all $n \in \N_0$. This, in turn, entails that if $\rho_{T_n^-}$ admits a density $p_{T_n^-}$ with respect to Lebesgue measure, then
\begin{equation*}
\int_{\R^m} \phi(x) \, \rho_{T_n}(\dd x) = \rho_{T_n^-}(\phi(\cdot + \Delta \nu_{T_n})) 
= \int_{\R^m}\phi(x + \Delta \nu_{T_n}) p_{T_n^-}(x) \, \dd x = \int_{\R^m} \phi(x) p_{T_n^-}(x - \Delta \nu_{T_n}) \, \dd x.
\end{equation*}
Therefore, since $\dC_b(\R^m)$ is a separating set (see, e.g., \citep[Chapter~3, Section~4]{ethierkurtz86:markov}), we have the equivalence of measures
$\rho_{T_n}(\dd x)$ and $p_{T_n^-}(x - \Delta \nu_{T_n}) \, \dd x$, implying that $\rho_{T_n}$ admits density with respect to Lebesgue measure on $\R^m$, given by $p_{T_n^-}(\cdot - \Delta \nu_{T_n})$.

We can now use the recursive structure of \eqref{eq:ZakaiSPDE} to get the claim. Define the process
\begin{equation*}
\nu^{(1)}_t \coloneqq \nu_t \ind_{t < T_1} + \nu_{T_1} \ind_{t \geq T_1}, \quad t \in [0,T],
\end{equation*}
and the random measure $\xi^{(1)}(\dd x) \coloneqq p_{0^-}(x - \Delta \nu_0) \, \dd x$, on $\R^m$.
Consider, for all $\phi \in \dC^2_b(\R^m)$, the Zakai equation
\begin{multline}\label{eq:Zakai1}
\rho^{(1)}_t(\phi) = \xi^{(1)}(\phi) + \int_0^t \rho^{(1)}_s\bigl(\bigl[\partial_s + \cA_s\bigr] \phi\bigr) \, \dd s 
\\
+ \int_0^t \rho^{(1)}_{s^-}\bigl(\dD_x \phi\bigr) \, \dd \nu^{(1)}_s + \int_0^t \gamma^{-1}(s) \rho^{(1)}_s(\phi h_s) \, \dd \overline B_s, \quad \bbP\text{-a.s.}, \quad t \in [0,T].
\end{multline}
Since $\nu^{(1)}$ satisfies point~\ref{hyp:nu} of Assumption~\ref{hyp:main}, we have that \eqref{eq:Zakai1} is the Zakai equation for the filtering problem of the partially observed system \eqref{eq:XSDE}--\eqref{eq:YSDEctrl}, with initial law $\xi^{(1)}$ and process $\nu^{(1)}$, which is continuous on $[0,T]$. Therefore, by Theorem~\ref{th:Zakaiuniqnucont}, $\rho^{(1)}$ is its unique solution and admits a density $p^{(1)}$ with respect to Lebesgue measure on $\R^m$, with $\bbE[\norm{p^{(1)}_t}^2_H] < +\infty$, for each $t \in [0,T]$. It is clear that, since $\nu_t = \nu_t^{(1)}$ on $\{t < T_1\}$, we have that $\rho_t = \rho_t^{(1)}$ on the same set, and hence $\rho_t$ admits density $p^{(1)}_t$ on $\{t < T_1\}$.

Next, let us define the process
\begin{equation*}
\nu^{(2)}_t \coloneqq \nu_{t+T_1} \ind_{t < T_2 - T_1} + \nu_{T_2} \ind_{t \geq T_2 - T_1}, \quad t \in [0,T],
\end{equation*}
and the random measure $\xi^{(2)}(\dd x) = p_{T_1^-}(x - \Delta \nu_{T_1}) \, \dd x$, on $\R^m$.
Consider, for all $\phi \in \dC^2_b(\R^m)$, the Zakai equation
\begin{multline}\label{eq:Zakai2new}
\rho^{(2)}_t(\phi) = \xi^{(2)}(\phi) + \int_0^t \rho^{(2)}_s\bigl(\bigl[\partial_s + \cA_{s+T_1}\bigr] \phi\bigr) \, \dd s 
\\
+ \int_0^t \rho^{(2)}_{s^-}\bigl(\dD_x \phi\bigr) \, \dd \nu^{(2)}_s + \int_0^t \gamma^{-1}(s+T_1) \rho^{(2)}_s(\phi h_{s+T_1}) \, \dd \overline B_{s+T_1}, \; \bbP\text{-a.s.}, \; t \in [0,T].
\end{multline}
Since $\nu^{(2)}$ satisfies point~\ref{hyp:nu} of Assumption \ref{hyp:main}, we have that \eqref{eq:Zakai2new} is the Zakai equation for the filtering problem of the partially observed system \eqref{eq:XSDE}--\eqref{eq:YSDEctrl}, with initial law $\xi^{(2)}$ and process $\nu^{(2)}$, which is continuous on $[0,T]$. Therefore, by Theorem~\ref{th:Zakaiuniqnucont}, $\rho^{(2)}$ is its unique solution and admits a density $p^{(2)}$ with respect to Lebesgue measure on $\R^m$, with $\bbE[\norm{p^{(2)}_t}^2_H] < +\infty$, for each $t \in [0,T]$. It is clear that, since $\nu_t = \nu^{(2)}_{t - T_1}$ on $\{T_1 \leq t < T_2\}$, we have that $\rho_t = \rho^{(2)}_{t - T_1}$ on the same set, and hence $\rho_t$ admits density $p^{(2)}_{t - T_1}$ on $\{T_1 \leq t < T_2\}$.

Continuing in this manner, we construct a sequence of solutions $(\rho^{(n)})_{n \in \N}$ and corresponding density processes $(p^{(n)})_{n \in \N}$. We deduce that the unnormalized filtering process is represented by
\begin{equation*}
\rho_t = \sum_{n=1}^\infty \rho^{(n)}_{t-T_n} \ind_{T_{n-1} \leq t < T_n}, \quad t \in [0,T],
\end{equation*}
and hence is the unique $\bbY$-adapted, c\`adl\`ag, $\cM_+(\R^m)$-valued solution to the Zakai equation \eqref{eq:Zakai}, admitting a $\bbY$-adapted, c\`adl\`ag, $H$-valued density process $p$, given by
\begin{equation*}
p_t = \sum_{n=1}^\infty p^{(n)}_{t-T_n} \ind_{T_{n-1} \leq t < T_n}, \quad t \in [0,T].
\end{equation*}
The fact that $\bbE[\norm{p_t}_H^2] < +\infty$, for all $t \in [0,T]$, follows from the analogous property for each of the processes $p^{(n)}$, $n \in \N$.
\end{proof}


\appendix
\section{Techincal results}\label{app:technical}

Let us recall that if $A$ (defined on a given filtered complete probability space) is a c\`adl\`ag, adapted, non-negative process, with $A_{0^-} = 0$, and $H$ is an optional process, satisfying $\int_0^t \abs{H_s} \, \dd A_s < +\infty$, for all $t \geq 0$, $\bbP$-a.s., then for any stopping time $\tau$ we have that
\begin{equation*}
\int_0^{\tau^-} H_s \, \dd A_s \coloneqq \int_0^{+\infty} H_s \ind_{s < \tau} \, \dd A_s.
\end{equation*}

\begin{lemma}\label{lem:gronwall}
Let $(\Omega, \cF, \bbF, \bbP)$ be a given filtered complete probability space, fix $T > 0$, and let $A$ and $H$ be two c\`adl\`ag, $\bbF$-adapted real-valued processes. Suppose that $A$ is non-decreasing, with $A_{0^-} = 0$ and $A_T \leq K$, $\bbP$-a.s., for some constant $K > 0$, and that $H$ satisfies one of the following:
\begin{enumerate}[label=\alph*.]
\item\label{hyp:Hgeneral} $\bbE[\sup_{t \in [0,T]} |H_{t^-}|] < +\infty$;
\item\label{hyp:Hpos} $H$ is non-negative and such that $\bbE[H_{t^-}] < +\infty$, for all $t \in [0,T]$.
\end{enumerate}
Assume, moreover, that for any $\bbF$-stopping time $\tau \leq T$ we have
\begin{equation}\label{eq:gronwallest}
\bbE[H_{\tau^-}] \leq M + \bbE\left[\int_0^{\tau^-} H_{s^-} \, \dd A_s\right],
\end{equation}
for some constant $M$. Then $\bbE[H_{T^-}] \leq M \de^K$.
\end{lemma}

\begin{proof}
The following reasoning is inspired by the proof of \citep[Lemma~IX.6.3]{jacod2013:limit}. Let us define
\begin{equation*}
\tilde A_t \coloneqq A_t \ind_{t < T} + K \ind_{t \geq T}, \quad t \geq 0.
\end{equation*}
$\tilde A$ is still a c\`adl\`ag, $\bbF$-adapted and non-decreasing process, with $\tilde A_{0^-} = 0$. Moreover, for any stopping time $\tau \leq T$, random measures $\ind_{s < \tau} \, \dd A_s$ and $\ind_{s < \tau} \, \dd \tilde A_s$ agree, therefore \eqref{eq:gronwallest} implies
\begin{equation}\label{eq:gronwallest2}
\bbE[H_{\tau^-}] \leq M + \bbE\left[\int_0^{\tau^-} H_{s^-} \, \dd \tilde A_s\right].
\end{equation}
Next, define $C_t \coloneqq \inf\{s \geq 0 \colon \tilde A_s \geq t\}$, $t \geq 0$, which (see, e.g., \citep[Chapter~VI, Def. 56]{dellacheriemeyer:B} or \citep[Proposition~I.1.28]{jacod2013:limit}) is an $\bbF$-stopping time for all $t \geq 0$, satisfying $C_t \leq T$, thanks to the definition of $\tilde A$.

We now fix $t \in [0, K]$. Using \eqref{eq:gronwallest2}, we get
\begin{equation*}
\bbE[H_{(C_t)^-}] \leq M + \bbE\left[\int_0^{+\infty} H_{s^-} \ind_{s < C_t} \, \dd \tilde A_s\right] = M + \bbE\left[\int_0^{+\infty} H_{(C_u)^-} \ind_{C_u < C_t} \, \dd u\right].
\end{equation*}
Since $C$ is a non-decreasing process, we have that $\{C_u < C_t\} \subset \{u < t\}$, and hence $\ind_{C_u < C_t} \leq \ind_{u < t}$. Therefore
\begin{equation*}
\bbE[H_{(C_t)^-}] \leq M + \bbE\left[\int_0^t H_{(C_u)^-} \, \dd u\right].
\end{equation*}
If $H$ satisfies condition \ref{hyp:Hpos} we can directly apply Fubini-Tonelli's theorem as below. If, instead, condition \ref{hyp:Hgeneral} holds, since $C_u \leq T$ and, for each fixed $\omega \in \Omega$, the image of the map $u \mapsto C_u(\omega)$ is a subset of $[0,T]$, we have that $\sup_{u \in [0,K]} |H_{(C_u)^-}| \leq \sup_{s \in [0,T]} |H_{s^-}|$, so
\begin{equation*}
\bbE\left[\int_0^t |H_{(C_u)^-}| \, \dd u\right] \leq K \, \bbE[\sup_{s \in [0,T]} |H_{s^-}|] < +\infty.
\end{equation*}
Therefore, we can apply Fubini-Tonelli's theorem and get
\begin{equation*}
\bbE[H_{(C_t)^-}] \leq M + \int_0^t \bbE[H_{(C_u)^-}] \, \dd u,
\end{equation*}
whence we obtain, from the usual Gronwall's lemma, $\bbE[H_{(C_t)^-}] \leq M \de^t$. Thanks to the definition of $\tilde A$, we have that $C_K = T$ and the claim follows letting $t = K$ in the last inequality. 
\end{proof}


\begin{proposition}\label{prop:etamart}
Under Assumptions \ref{hyp:main} and \ref{hyp:h}, the process $\eta$, defined in \eqref{eq:eta}, is a $(\bbP,\bbF)$-martingale.
\end{proposition}
\begin{proof}
Let us notice, first, a fact that will be useful in this proof. It can be easily shown that condition \eqref{eq:gammaunifpd} implies, for some constant $C_\gamma$,
\begin{equation}\label{eq:invgammanorm}
\norm{\gamma^{-1}(t)} \leq C_\gamma, \quad \forall t \in [0,T].
\end{equation}

Let us define, for all $t \in [0,T]$, 
\begin{equation*}
Z_t \coloneqq \int_0^t \gamma^{-1}(s) h(s, X_s) \, \dd \overline B_s.
\end{equation*}
Thanks to condition \eqref{eq:hlin} and using \eqref{eq:invgammanorm} and \eqref{eq:Xestimate}, we easily get
\begin{align}
&\phantom{\mathop{\leq}} \bbE\left[\int_0^T \norm{\gamma^{-1}(s) h(s, X_s)}^2 \, \dd s\right]
\leq n \bbE\left[\int_0^T \norm{\gamma^{-1}(s)}^2 \norm{h(s, X_s)}^2 \, \dd s\right] \notag
\\
&\leq n C_h C_\gamma \bbE\left[\int_0^T (1+\norm{X_s}^2) \, \dd s\right]
\leq n C_h C_\gamma T[1 + \kappa(1+ \bbE[\norm{X_{0^-}}^2])] < +\infty. \label{eq:gammainvhintegr}
\end{align}
Therefore, $Z$ is an $(\bbF,\bbP)$-martingale, and hence $\eta$, which is the Doléans-Dade exponential of $Z$, is a non-negative local $(\bbF,\bbP)$-martingale (see, e.g., \citep[Lemma~15.3.2]{cohen:stochcalculus}). Thus, to prove the claim it is enough to show that $\bbE[\eta_t]=1$ for all $t \in [0,T]$.

We start proving, first, that $\bbE[\eta_t \norm{X_{t^-}}^2] \leq C$, for all $t \in [0,T]$, where $C$ is an appropriately chosen constant. For the sake of brevity, let us write $b_s \coloneqq b(s, X_s)$, $\sigma_s \coloneqq \sigma(s, X_s)$, and $h_s \coloneqq h(s, X_s)$.
Applying It\^o's formula we get
\begin{align*}
\norm{X_t}^2 
&= \norm{X_{0^-}}^2 + \int_0^t \left[2X_{s^-}^* b_s + \norm{\sigma_s}^2\right] \, \dd s + 2\int_0^t X_{s^-}^* \sigma_s \, \dd W_s
\\
&+ 2\int_0^t X_{s^-} \, \dd \nu_s + \sum_{0 \leq s \leq t} \{\norm{X_s}^2 - \norm{X_{s^-}}^2 - 2 X_{s^-} \cdot \Delta \nu_s\},
\end{align*}
and using the integration by parts rule we have
\begin{align*}
\eta_t \norm{X_t}^2 
&= \norm{X_{0^-}}^2 + \int_0^t \left[2\eta_{s^-} X_{s^-}^* b_s + \eta_{s^-}\norm{\sigma_s}^2\right] \, \dd s + 2\int_0^t \eta_{s^-} X_{s^-}^* \sigma_s \, \dd W_s
\\
& + \int_0^t \norm{X_{s^-}}^2 \eta_s \gamma^{-1}(s) h_s \, \dd \overline B_s + 2 \int_0^t \eta_{s^-} X_{s^-} \, \dd \nu_s 
\\
&+ \sum_{0 \leq s \leq t} \eta_{s^-} \{\norm{X_s}^2 - \norm{X_{s^-}}^2 - 2 X_{s^-} \cdot \Delta \nu_s\}.
\end{align*}
Therefore, for any fixed $\epsilon > 0$, we obtain
\begin{align}
&\mathop{\phantom{+}}\dfrac{\eta_t \norm{X_t}^2}{1+\epsilon \eta_t \norm{X_t}^2}
= \dfrac{\norm{X_{0^-}}^2}{1+\epsilon \norm{X_{0^-}}^2} + \int_0^t \dfrac{\eta_{s^-}}{[1+\epsilon \eta_{s^-} \norm{X_{s^-}}^2]^2}\left[2 X_{s^-}^* b_s + \norm{\sigma_s}^2\right] \, \dd s \notag
\\
&- \int_0^t \dfrac{\epsilon \eta_{s^-}^2}{[1+\epsilon \eta_{s^-} \norm{X_{s^-}}^2]^3}\left[4 \norm{X_{s^-}^* \sigma_s}^2 + \norm{X_{s^-}}^4 \norm{\gamma^{-1}(s) h_s}^2\right] \, \dd s \notag
\\
&+ \int_0^t \dfrac{2\eta_{s^-}}{[1+\epsilon \eta_{s^-} \norm{X_{s^-}}^2]^2} X_{s^-} \, \dd \nu^c_s
+\int_0^t \dfrac{2\eta_{s^-}}{[1+\epsilon \eta_{s^-} \norm{X_{s^-}}^2]^2} X_{s^-}^* \sigma_s \, \dd W_s \notag
\\
&+ \int_0^t \dfrac{\eta_{s^-}\norm{X_{s^-}}^2}{[1+\epsilon \eta_{s^-} \norm{X_{s^-}}^2]^2} \gamma^{-1}(s) h_s \, \dd \overline B_s
+ \sum_{0 \leq s \leq t}\biggl\{\dfrac{\eta_s \norm{X_s}^2}{1+\epsilon \eta_s \norm{X_s}^2} - \dfrac{\eta_{s^-} \norm{X_{s^-}}^2}{1+\epsilon \eta_{s^-} \norm{X_{s^-}}^2}\biggr\}, \label{eq:etanormXdiff}
\end{align}
where $\nu^c$ denotes the continuous part of the process $\nu$.

With standard estimates (see, e.g., \citep[Solution to Exercise~3.11]{bain:fundofstochfilt}) it is possible to show that the stochastic integrals with respect to Brownian motions $W$ and $\overline B$ are $(\bbF, \bbP)$-martingales. This implies, thanks to the optional sampling theorem, that these stochastic integrals have zero expectation even when evaluated at any bounded stopping time. Fixing a $\bbF$-stopping time $\tau \leq t$, for arbitrary $t \in [0,T]$, taking the expectation and noticing that the third term in \eqref{eq:etanormXdiff} is non-negative, we get
\begin{align}
&\bbE\left[\dfrac{\eta_{\tau^-} \norm{X_{\tau^-}}^2}{1+\epsilon \eta_{\tau^-} \norm{X_{\tau^-}}^2}\right]
\leq \bbE\left[\dfrac{\norm{X_{0^-}}^2}{1+\epsilon \norm{X_{0^-}}^2}\right]
\notag
\\
+ &\bbE\biggl[\int_0^{\tau^-} \dfrac{\eta_{s^-}\left[2 X_{s^-}^* b_s + \norm{\sigma_s}^2\right]}{[1+\epsilon \eta_{s^-} \norm{X_{s^-}}^2]^2} \, \dd s\biggr] \notag
+ \bbE\biggl[\int_0^{\tau^-} \dfrac{2\eta_{s^-}X_{s^-} }{[1+\epsilon \eta_{s^-} \norm{X_{s^-}}^2]^2} \, \dd \nu^c_s\biggr]
\\
+ &\bbE\biggl[\,\sum_{0 \leq s < \tau}\biggl\{\dfrac{\eta_s \norm{X_s}^2}{1+\epsilon \eta_s \norm{X_s}^2} - \dfrac{\eta_{s^-} \norm{X_{s^-}}^2}{1+\epsilon \eta_{s^-} \norm{X_{s^-}}^2}\biggr\}\biggr]. \label{eq:etanormXmean}
\end{align}
We proceed, now, to find suitable estimates for the terms appearing in \eqref{eq:etanormXmean}.

Notice that, thanks to conditions \eqref{eq:blip} and \eqref{eq:sigmalip}, we have that for some constant $C_1$
\begin{equation*}
\left|2 X_{s^-}^* b_s + \norm{\sigma_s}^2\right| \leq C_1(1+\norm{X_{s^-}}^2), \quad \bbP\text{-a.s.}, \quad s \in [0,T],
\end{equation*}
Recalling that $\eta$ is non-negative and that $\bbE[\eta_t] \leq 1$, for any $t \in [0,T]$, we get
\begin{align}
&\mathop{\phantom{\leq}} \bbE\biggl[\int_0^{\tau^-} \dfrac{\eta_{s^-}\left[2 X_{s^-}^* b_s + \norm{\sigma_s}^2\right]}{[1+\epsilon \eta_{s^-} \norm{X_{s^-}}^2]^2} \, \dd s\biggr]
\leq C_1 \bbE\biggl[\int_0^{\tau^-} \dfrac{\eta_{s^-}(1+\norm{X_{s^-}}^2)}{[1+\epsilon \eta_{s^-} \norm{X_{s^-}}^2]^2} \, \dd s \biggr]
 \notag
\\
&\leq C_1 \bbE\biggl[\int_0^{\tau^-} \eta_{s^-} \, \dd s \biggr]
+ C_1 \bbE\biggl[\int_0^{\tau^-} \dfrac{\eta_{s^-}\norm{X_{s^-}}^2}{1+\epsilon \eta_{s^-} \norm{X_{s^-}}^2} \, \dd s \biggr]
\notag
\\
&\leq C_1 T
+ C_1 \bbE\biggl[\int_0^{\tau^-}\dfrac{\eta_{s^-}\norm{X_{s^-}}^2}{1+\epsilon \eta_{s^-} \norm{X_{s^-}}^2} \, \dd s \biggr]. \label{eq:etaXdriftest}
\end{align}

Next, we see that
\begin{align}
&\mathop{\phantom{=}} \bbE\biggl[\int_0^{\tau^-} \dfrac{2\eta_{s^-} X_{s^-} }{[1+\epsilon \eta_{s^-} \norm{X_{s^-}}^2]^2}\, \dd \nu^c_s\biggr]
= \sum_{i=1}^m \bbE\biggl[\int_0^{\tau^-} \dfrac{2\eta_{s^-} X_{s^-}^i }{[1+\epsilon \eta_{s^-} \norm{X_{s^-}}^2]^2} \, \dd \nu^{i,c}_s\biggr] 
\notag
\\
&\leq \sum_{i=1}^m \bbE\biggl[\int_0^{\tau^-} \dfrac{2\eta_{s^-} \abs{X_{s^-}^i} }{[1+\epsilon \eta_{s^-} \norm{X_{s^-}}^2]^2} \, \dd \abs{\nu^{i,c}}_s\biggr]
\leq \sum_{i=1}^m \bbE\biggl[\int_0^{\tau^-} \dfrac{\eta_{s^-}(1+|X_{s^-}^i|^2)}{[1+\epsilon \eta_{s^-} \norm{X_{s^-}}^2]^2} \, \dd \abs{\nu^{i,c}}_s\biggr] \notag
\\
&\leq \sum_{i=1}^m \bbE\biggl[\int_0^{\tau^-} \eta_{s^-} \, \dd \abs{\nu^{i,c}}_s\biggr]
+ \sum_{i=1}^m \bbE\biggl[\int_0^{\tau^-} \dfrac{\eta_{s^-}\norm{X_{s^-}}^2}{1+\epsilon \eta_{s^-} \norm{X_{s^-}}^2} \, \dd \abs{\nu^{i,c}}_s\biggr]. \label{eq:etaXnuest}
\end{align}
Similarly to what we did in the proof of Lemma~\ref{lem:gronwall}, let us define
\begin{equation*}
\tilde \nu^i_t \coloneqq \abs{\nu^i}_t \ind_{t < T} + K \ind_{t \geq T}, \quad t \geq 0, \quad i = 1, \dots, m.
\end{equation*}
For each $i = 1, \dots, m$, $\tilde \nu^i$ is a $\bbY$-adapted, c\`adl\`ag, non-decreasing process, with $\tilde \nu^i_{0^-}=0$. Moreover, random measures $\ind_{s < \tau} \, \dd \abs{\nu^i}_s$ and $\ind_{s < \tau} \, \dd \tilde \nu^i_s$ agree, therefore
\begin{equation*}
\bbE\biggl[\int_0^{\tau^-} \eta_{s^-} \, \dd \abs{\nu^i}_s\biggr] = \bbE\biggl[\int_0^{\tau^-} \eta_{s^-} \, \dd \tilde \nu^i_s\biggr], \quad i = 1, \dots, m,
\end{equation*}
and, in particular,
\begin{equation*}
\bbE\biggl[\int_0^{\tau^-} \eta_{s^-} \, \dd \abs{\nu^{i,c}}_s\biggr] = \bbE\biggl[\int_0^{\tau^-} \eta_{s^-} \, \dd \tilde \nu^{i,c}_s\biggr], \quad i = 1, \dots, m.
\end{equation*}
Let us define the changes of time $C^i_t \coloneqq \inf\{s \geq 0 \colon \tilde \nu^i_s \geq t\}$, for all $t \geq 0$ and all $i = 1, \dots, m$. Then, noticing that $\{C_s^i \leq t\} = \{\tilde \nu_t^i \geq s\}$ and recalling that $\eta$ is non-negative and $\tilde \nu^i_T = K$, we get
\begin{align*}
&\mathop{\phantom{=}}\bbE\biggl[\int_0^{\tau^-} \eta_{s^-} \, \dd \tilde \nu^{i,c}_s\biggr]
\leq  \bbE\left[\int_0^T \eta_{s^-} \, \dd \tilde \nu^{i,c}_s\right]
\leq  \bbE\left[\int_0^T \eta_{s^-} \, \dd \tilde \nu^i_s\right]
= \bbE\left[\int_0^{+\infty} \eta_{(C^i_s)^-} \ind_{C_s^i \leq T} \, \dd s\right]
\\
&= \bbE\left[\int_0^{+\infty} \eta_{(C^i_s)^-} \ind_{s \leq \tilde \nu_T^i} \, \dd s\right]
= \bbE\left[\int_0^K \eta_{(C^i_s)^-} \, \dd s\right]
=\int_0^K \bbE[\eta_{(C^i_s)^-}] \, \dd s.
\end{align*}
Since $\bbE[\eta_t] \leq 1$, for any $t \in [0,T]$, and $C^i_s \leq T$, for all $s \in [0,K]$, we get that
\begin{equation}\label{eq:etanuest}
\bbE\biggl[\int_0^{\tau^-} \eta_{s^-} \, \dd \abs{\nu^{i,c}}_s\biggr] = \bbE\biggl[\int_0^{\tau^-} \eta_{s^-} \, \dd \tilde \nu^{i,c}_s\biggr] \leq K, \quad i = 1, \dots, m.
\end{equation}
Similarly, we obtain also
\begin{equation}\label{eq:etanuest2}
\bbE\biggl[\int_0^{\tau^-} \eta_{s^-} \, \dd \abs{\nu^i}_s\biggr] = \bbE\biggl[\int_0^{\tau^-} \eta_{s^-} \, \dd \tilde \nu^i_s\biggr] \leq K, \quad i = 1, \dots, m.
\end{equation}
Therefore, putting together \eqref{eq:etaXnuest} and \eqref{eq:etanuest} we obtain
\begin{equation}\label{eq:etaXnuest2}
\bbE\biggl[\int_0^{\tau^-} \dfrac{2\eta_{s^-}}{[1+\epsilon \eta_{s^-} \norm{X_{s^-}}^2]^2} X_{s^-} \, \dd \nu^c_s\biggr]
\leq mK + \sum_{i=1}^m \bbE\biggl[\int_0^{\tau^-} \dfrac{\eta_{s^-}\norm{X_{s^-}}^2}{1+\epsilon \eta_{s^-} \norm{X_{s^-}}^2} \, \dd \abs{\nu^{i,c}}_s\biggr]
\end{equation}

We are left with estimating the last term of \eqref{eq:etanormXmean}. We have:
\begin{align*}
&\mathop{\phantom{=}} \bbE\biggl[\,\sum_{0 \leq s < \tau}\biggl\{\dfrac{\eta_s \norm{X_s}^2}{1+\epsilon \eta_s \norm{X_s}^2} - \dfrac{\eta_{s^-} \norm{X_{s^-}}^2}{1+\epsilon \eta_{s^-} \norm{X_{s^-}}^2}\biggr\}\biggr]
\leq \bbE\biggl[\,\sum_{0 \leq s < \tau}\biggl\{\dfrac{\eta_{s^-}( \norm{X_s}^2 - \norm{X_{s^-}}^2)}{1+\epsilon \eta_{s^-} \norm{X_{s^-}}^2}\biggr\}\biggr]
\\
&= \bbE\biggl[\,\sum_{0 \leq s < \tau}\biggl\{\dfrac{\eta_{s^-}(\norm{\Delta \nu_s}^2 + 2 X_{s^-} \cdot \Delta \nu_s)}{1+\epsilon \eta_{s^-} \norm{X_{s^-}}^2} \biggr\}\biggr]
\leq \bbE\biggl[\,\sum_{0 \leq s < \tau}\sum_{i=1}^m \left\{\dfrac{\eta_{s^-}(\abs{\Delta \nu^i_s} + 1 + \abs{X_{s^-}^i}^2)}{1+\epsilon \eta_{s^-} \norm{X_{s^-}}^2}\abs{\Delta \nu^i_s} \right\} \biggr],
\end{align*}
where we used the fact that $\eta$ is continuous. Since all quantities in the last term are non negative and $\abs{\Delta \nu^i_s} \leq K$, for all $s \in [0,T]$ and all $i = 1, \dots, m$, $\bbP$-a.s., we get that
\begin{align}
&\mathop{\phantom{\leq}} \bbE\biggl[\,\sum_{0 \leq s < \tau}\sum_{i=1}^m \left\{\dfrac{\eta_{s^-}(\abs{\Delta \nu^i_s} + 1 + \abs{X_{s^-}^i}^2)}{1+\epsilon \eta_{s^-} \norm{X_{s^-}}^2}\abs{\Delta \nu^i_s} \right\} \biggr] \notag
\\
&\leq (1+K) \sum_{i=1}^m \bbE\biggl[\sum_{0 \leq s < \tau} \eta_{s^-} \abs{\Delta \nu^i_s} \biggr] + \sum_{i=1}^m \bbE\biggl[\,\sum_{0 \leq s < \tau}  \left\{\dfrac{\eta_{s^-}\norm{X_{s^-}}^2}{1+\epsilon \eta_{s^-} \norm{X_{s^-}}^2}\abs{\Delta \nu^i_s} \right\} \biggr] \notag
\\
&\leq (1+K) \sum_{i=1}^m \bbE\biggl[\int_0^{\tau^-} \eta_{s^-} \dd \abs{\nu^i}_s \biggr] + \sum_{i=1}^m \bbE\biggl[\,\sum_{0 \leq s < \tau}  \left\{\dfrac{\eta_{s^-}\norm{X_{s^-}}^2}{1+\epsilon \eta_{s^-} \norm{X_{s^-}}^2}\Delta \abs{\nu^i}_s \right\} \biggr] \notag
\\
&\leq mK(1+K) + \sum_{i=1}^m \bbE\biggl[\,\sum_{0 \leq s < \tau}  \left\{\dfrac{\eta_{s^-}\norm{X_{s^-}}^2}{1+\epsilon \eta_{s^-} \norm{X_{s^-}}^2}\Delta \abs{\nu^i}_s \right\} \biggr], \label{eq:etaXjumpest}
\end{align}
where we used \eqref{eq:etanuest2} and the fact that $\abs{\Delta \nu^i} = \Delta \abs{\nu^i}$.

Therefore, feeding \eqref{eq:etaXdriftest}, \eqref{eq:etaXnuest2}, and \eqref{eq:etaXjumpest} back into \eqref{eq:etanormXmean}, we obtain
\begin{align}
&\mathop{\phantom{\leq}} \bbE\left[\dfrac{\eta_{\tau^-} \norm{X_{\tau^-}}^2}{1+\epsilon \eta_{\tau^-} \norm{X_{\tau^-}}^2}\right]
\leq M\biggl\{1 + \bbE\biggl[\int_0^{\tau^-} \dfrac{\eta_{s^-}\norm{X_{s^-}}^2}{1+\epsilon \eta_{s^-} \norm{X_{s^-}}^2} \dd A_s \biggr]\biggr\},
\end{align}
where $M$ is a suitable constant, not depending on $\epsilon$, and $A$ is the process
\begin{equation*}
A_t \coloneqq t + \sum_{i=1}^m \abs{\nu^i}_t, \quad t \in [0,T].
\end{equation*}
Clearly, $A$ is a c\`adl\`ag, $\bbY$- (and hence $\bbF$-) adapted, non-negative process, with $A_{0^-} = 0$ and $A_T \leq T+mK$. Moreover, $\frac{\eta_{t^-}\norm{X_{t^-}}^2}{1+\epsilon \eta_{t^-} \norm{X_{t^-}}^2} \leq \frac 1\epsilon$, for all $t \in [0,T]$, $\bbP$-a.s.
Therefore, we can apply Lemma~\ref{lem:gronwall} and obtain
\begin{equation*}
\bbE\left[\dfrac{\eta_{t^-} \norm{X_{t^-}}^2}{1+\epsilon \eta_{t^-} \norm{X_{t^-}}^2}\right] \leq M \de^{M(T+mK)}.
\end{equation*}
Recalling that $\eta$ is continuous we get, applying Fatou's lemma,
\begin{equation}\label{eq:etanormXest}
\bbE[\eta_t \norm{X_{t^-}}^2] = \bbE\left[\lim_{\epsilon \to 0} \dfrac{\eta_t \norm{X_{t^-}}^2}{1+\epsilon \eta_t \norm{X_{t^-}}^2}\right]
\leq \liminf_{\epsilon \to 0} \bbE\left[\dfrac{\eta_t \norm{X_{t^-}}^2}{1+\epsilon \eta_t \norm{X_{t^-}}^2}\right] \leq M \de^{M(T+mK)}.
\end{equation} 
It is important to stress that \eqref{eq:etanormXest} holds for any $t \in [0,T]$, since $t$ was arbitrarily chosen.

Now we can finally obtain that $\bbE[\eta_t]=1$, for all $t \in [0,T]$. By It\^o's formula, for an arbitrarily fixed $\epsilon > 0$ and all $t \in [0,T]$,
\begin{equation*}
\dfrac{\eta_t}{1+\epsilon \eta_t} = \dfrac{1}{1+\epsilon} - \int_0^t \dfrac{\epsilon \eta_s^2}{(1+\epsilon \eta_s)^3} \norm{\gamma^{-1}(s) h(s, X_s)}^2 \, \dd s + \int_0^t \dfrac{\eta_s}{(1+\epsilon\eta_s)^2} \gamma^{-1}(s) h(s, X_s) \, \dd \overline B_s.
\end{equation*}
Thanks to conditions \eqref{eq:hlin} and \eqref{eq:gammaunifpd}, standard computations show that the stochastic integral is a $(\bbP,\bbF)$-martingale. Therefore, taking the expectation we get
\begin{equation*}
\bbE\left[\dfrac{\eta_t}{1+\epsilon \eta_t}\right] = \dfrac{1}{1+\epsilon} - \bbE\left[\int_0^t \dfrac{\epsilon \eta_s^2}{(1+\epsilon \eta_s)^3} \norm{\gamma^{-1}(s) h(s, X_s)}^2 \, \dd s\right].
\end{equation*}
Notice that $\frac{\epsilon \eta_s^2}{(1+\epsilon \eta_s)^3} \norm{\gamma^{-1}(s) h(s, X_s)}^2 \longrightarrow 0$, as $\epsilon \to 0$, $\dd \bbP \otimes \dd t$-a.s. Moreover,
\begin{equation*}
\dfrac{\epsilon \eta_s^2}{(1+\epsilon \eta_s)^3} \norm{\gamma^{-1}(s) h(s, X_s)}^2 \leq \eta_s\norm{\gamma^{-1}(s) h(s, X_s)}^2, \quad s \in [0,T], 
\end{equation*}
that, using conditions \eqref{eq:hlin} and \eqref{eq:gammaunifpd}, satisfies (see also \eqref{eq:gammainvhintegr})
\begin{align*}
&\phantom{\mathop{\leq}} \bbE\left[\int_0^T \eta_s \norm{\gamma^{-1}(s) h(s, X_s)}^2 \, \dd s\right]
\leq n C_h C_\gamma \bbE\left[\int_0^T \eta_s (1+\norm{X_s}^2) \, \dd s\right]
\\
&= n C_h C_\gamma\left\{\int_0^T \bbE[\eta_s] \, \dd s + \int_0^T \bbE[\eta_s \norm{X_{s^-}}^2] \, \dd s\right\}
\leq n C_h C_\gamma T[1 + M \de^{M(T+mK)}],
\end{align*}
where we used the fact that $\bbE[\eta_t] \leq 1$, for all $t \in [0,T]$, and \eqref{eq:etanormXest}.
Similarly, $\frac{\eta_t}{1+\epsilon \eta_t} \to \eta_t$, as $\epsilon \to 0$, $\dd \bbP \otimes \dd t$-a.s., and $\bbE[\int_0^T \eta_s \, \dd s] \leq T$. Therefore, by the dominated convergence theorem
\begin{equation*}
\bbE[\eta_t] = \lim_{\epsilon \to 0} \bbE\left[\dfrac{\eta_t}{1+\epsilon \eta_t}\right] = \lim_{\epsilon \to 0} \left\{ \dfrac{1}{1+\epsilon} - \bbE\left[\int_0^t \dfrac{\epsilon \eta_s^2}{(1+\epsilon \eta_s)^3} \norm{\gamma^{-1}(s) h(s, X_s)}^2 \, \dd s\right]\right\} = 1,
\end{equation*}
and this concludes the proof.
\end{proof}


\begin{proof}[Proof of Lemma~\ref{lem:tousegronwall}]
Fix $\epsilon > 0$.
To start, let us notice that continuity of process $\nu$ implies that also $\zeta$ is continuous and, therefore, $\zeta_t = \zeta_{t^-}$ and $T_\epsilon \zeta_t = T_\epsilon \zeta_{t^-}$, $\dd t \otimes \dd \bbP$-almost everywhere.

Since $\psi_{2\epsilon}$ is bounded by $(4\pi\epsilon)^{-\frac m2}$, we get that for all $t \in [0,T]$,
\begin{align}\label{eq:normZtest}
\norm{T_\epsilon \zeta_t}_H^2 
&= \int_{\R^m} \biggl[\int_{\R^m} \psi_\epsilon(x-y) \, \zeta_t(\dd x)\biggr]^2 \, \dd y \\
&= \int_{\R^m} \int_{\R^m} \int_{\R^m} \psi_{\epsilon}(x-y) \psi_{\epsilon}(z-y) \, \zeta_t(\dd x) \, \zeta_t(\dd z) \, \dd y \\
&= \int_{\R^m} \int_{\R^m} \psi_{2\epsilon}(x-z) \, \zeta_t(\dd x) \, \zeta_t(\dd z)
\leq (4\pi\epsilon)^{-\frac m2} \zeta_t(\one)^2.
\end{align}
Taking into account \eqref{eq:Zakai} and the fact that $\nu$ is continuous, the process $\zeta(\one)$ satisfies
\begin{equation*}
\zeta_t(\one) = 1 + \int_0^t \gamma^{-1}(s) \zeta_s(h_s) \, \dd \overline B_s, \quad t \in [0,T],
\end{equation*}
where $h_t(\cdot) \coloneqq h(t,\cdot)$, $t \in [0,T]$. Thanks to Assumption \ref{hyp:bddcoeff}, $\zeta_t(h_t) < +\infty$, $\bbP$-a.s., for all $t \in [0,T]$. Therefore, since $\zeta_t$ is $\bbP$-a.s. a finite (non-negative) measure, for any $t \in [0,T]$, we get that $\zeta(\one)$ is a non-negative $(\bbP,\bbY)$-local martingale, and hence a $(\bbP,\bbY)$-supermartingale.

The next step is to prove that $\zeta(\one)$ is a square-integrable\footnote{If $M = (M_t)_{t \in [0,T]}$ is any martingale, we say that $M$ is a $p$-integrable martingale, with $p \geq 1$, if $\bbE[\sup_{t \in [0,T]} \abs{M_t}^p]^{1/p} < +\infty$.} $(\bbP,\bbY)$-martingale. We follow, first, a reasoning analogous to that of \citep[Lemma~4.3.1]{bensoussan:stochcontrol} (see also \citep[Lemma~3.29]{bain:fundofstochfilt}) to provide an explicit representation of $\zeta(\one)$.
By It\^o's formula we obtain, for any $\delta > 0$ and all $t \in [0,T]$,
\begin{align}\label{eq:itozeta1}
&\mathop{\phantom{=}}\log\left(\sqrt{\delta + \zeta_t(\one)^2}\right) 
= \log\left(\sqrt{1+\delta}\right) + \int_0^t \dfrac{\zeta_s(\one)}{\delta + \zeta_s(\one)^2} \gamma^{-1}(s) \zeta_s(h_s) \, \dd \overline B_s \notag \\
&+ \dfrac 12 \int_0^t \dfrac{\delta - \zeta_s(\one)^2}{[\delta + \zeta_s(\one)^2]^2} \sum_{i=1}^n \left(\sum_{j=1}^n \gamma^{-1}_{ij}(s) \zeta_s(h_s^j)\right)^2 \, \dd s.
\end{align}
Since, thanks to Assumption \ref{hyp:bddcoeff} and \eqref{eq:invgammanorm}, 
\begin{equation}\label{eq:gammazetahest}
\sum_{i=1}^n \biggl(\sum_{j=1}^n \gamma^{-1}_{ij}(t) \zeta_t(h_t^j)\biggr)^2 \leq (n C_\gamma K_h \zeta_t(\one))^2, \quad \bbP\text{-a.s.}, \quad \forall t \in [0,T],
\end{equation}
and
$\frac{\delta - \zeta_t(\one)^2}{[\delta + \zeta_t(\one)^2]^2} \leq \frac{1}{\delta + \zeta_t(\one)^2}$, $\bbP$-a.s., for all $t \in [0,T]$,
we have
\begin{equation*}
\dfrac{\zeta_s(\one)^2}{[\delta + \zeta_s(\one)^2]^2} \sum_{i=1}^n \left(\sum_{j=1}^n \gamma^{-1}_{ij}(t) \zeta_t(h_t^j)\right)^2 \leq \left(n C_\gamma K_h\dfrac{\zeta_s(\one)^2}{\delta + \zeta_s(\one)^2} \right)^2 \leq (nC_\gamma K_h)^2, \quad \forall t \in [0,T],
\end{equation*}
and
\begin{equation*}
\dfrac{\delta - \zeta_s(\one)^2}{[\delta + \zeta_s(\one)^2]^2} \sum_{i=1}^n \left(\sum_{j=1}^n \gamma^{-1}_{ij}(s) \zeta_s(h_s^j)\right)^2 \leq \dfrac{\zeta_s(\one)^2}{\delta + \zeta_s(\one)^2} (n C_\gamma K_h)^2 \leq (nC_\gamma K_h)^2, \quad \forall t \in [0,T].
\end{equation*}
Both the r.h.s.\ of the last two inequalities are integrable on $[0,T]$, therefore we can pass to the limit, as $\delta \to 0$, in \eqref{eq:itozeta1}, getting that, for all $t \in [0,T]$,
\begin{equation}\label{eq:logzeta1}
\log(\zeta_t(\one)) 
= 1 + \int_0^t \gamma^{-1}(s) \zeta^1_s(h_s) \, \dd \overline B_s \\
- \dfrac 12 \int_0^t \sum_{i=1}^n \biggl(\sum_{j=1}^n \gamma^{-1}_{ij}(s) \zeta^1_s(h_s^j)\biggr)^2 \, \dd s,
\end{equation}
where $\zeta^1_t(\dd x) \coloneqq \frac{\zeta_t(\dd x)}{\zeta_t(\one)}$, $t \in [0,T]$, is the normalized process associated to $\zeta$. From \eqref{eq:logzeta1} we get the explicit representation for $\zeta(\one)$, i.e., for all $t \in [0,T]$,
\begin{equation}\label{eq:zeta1}
\zeta_t(\one) = \exp\biggl\{\int_0^t \gamma^{-1}(s) \zeta^1_s(h_s) \, \dd \overline B_s \\
- \dfrac 12 \int_0^t \sum_{i=1}^n \biggl(\sum_{j=1}^n \gamma^{-1}_{ij}(s) \zeta^1_s(h_s^j)\biggr)^2 \, \dd s\biggr\}.
\end{equation}
This entails that $\zeta(\one)$ coincides with the Doléans-Dade exponential of the continuous $(\bbP,\bbY)$-local martingale
$\int_0^t \gamma^{-1}(s) \zeta^1_s(h_s) \, \dd \overline B_s$, $t \in [0,T]$.
Using once more \eqref{eq:gammazetahest} we have that, for any $k > 1$,
\begin{equation*}
\bbE\biggl[\exp\biggl\{\dfrac k2 \int_0^T \sum_{i=1}^n \biggl(\sum_{j=1}^n \gamma^{-1}_{ij}(t) \zeta^1_t(h_t^j)\biggr)^2 \dd t\biggr\}\biggr] \leq \exp\left\{\dfrac{kT(n C_\gamma K_h)^2}{2}\right\}.
\end{equation*}
Applying \citep[Theorem~15.4.6]{cohen:stochcalculus}, we get that, for any $p>1$, $\zeta(\one)$ is a $p$-integrable (in particular, square-integrable) $(\bbP,\bbY)$-martingale. Therefore, from \eqref{eq:normZtest} we get
\begin{equation*}
\bbE[\sup_{t \in [0,T]} \norm{T_\epsilon \zeta_t}_H^2] \leq (4\pi\epsilon)^{-\frac m2} \bbE[\sup_{t \in [0,T]} \zeta_t(\one)^2] < +\infty,
\end{equation*}
whence, recalling the remark at the beginning of the proof, the claim.
\end{proof}

\bibliographystyle{plainnat}
\bibliography{Bibliography}

\begin{thebibliography}{37}
\providecommand{\natexlab}[1]{#1}
\providecommand{\url}[1]{\texttt{#1}}
\expandafter\ifx\csname urlstyle\endcsname\relax
  \providecommand{\doi}[1]{doi: #1}\else
  \providecommand{\doi}{doi: \begingroup \urlstyle{rm}\Url}\fi

\bibitem[Alvarez and Shepp(1998)]{AlvarezShepp}
L.H.R. Alvarez and L.A. Shepp.
\newblock Optimal harvesting of stochastically fluctuating populations.
\newblock \emph{J. Math. Biol.}, 37:\penalty0 155--177, 1998.

\bibitem[Bain and Crisan(2009)]{bain:fundofstochfilt}
A.~Bain and D.~Crisan.
\newblock \emph{Fundamentals of Stochastic Filtering}.
\newblock Springer, New York, 2009.

\bibitem[Bandini et~al.(2019)Bandini, Cosso, Fuhrman, and
  Pham]{bandini2019:wasserstein}
E.~Bandini, A.~Cosso, M.~Fuhrman, and H.~Pham.
\newblock Randomized filtering and {B}ellman equation in {W}asserstein space
  for partial observation control problem.
\newblock \emph{Stochastic Process. Appl.}, 129\penalty0 (2):\penalty0
  674--711, 2019.

\bibitem[Bandini et~al.(2021)Bandini, Calvia, and Colaneri]{bandini2021:filt}
E.~Bandini, A.~Calvia, and K.~Colaneri.
\newblock Stochastic filtering of a pure jump process with predictable jumps
  and path-dependent local characteristics.
\newblock Preprint, \href{https://arxiv.org/abs/2004.12944}{arXiv:2004.12944},
  2021.

\bibitem[Bensoussan(1992)]{bensoussan:stochcontrol}
A.~Bensoussan.
\newblock \emph{Stochastic control of partially observable systems}.
\newblock Cambridge University Press, Cambridge, 1992.

\bibitem[Brémaud(1981)]{bremaud:pp}
P.~Brémaud.
\newblock \emph{Point Processes and Queues}.
\newblock Springer Series in Statistics. Springer-Verlag, New York, 1981.

\bibitem[Callegaro et~al.(2020)Callegaro, Ceci, and Ferrari]{Callegaro-etal}
G.~Callegaro, C.~Ceci, and G.~Ferrari.
\newblock Optimal reduction of public debt under partial observation of the
  economic growth.
\newblock \emph{Fin. Stoch.}, 24\penalty0 (4):\penalty0 1083--1132, 2020.

\bibitem[Calvia(2020)]{calvia:filtcontrol}
A.~Calvia.
\newblock Stochastic filtering and optimal control of pure jump {M}arkov
  processes with noise-free partial observation.
\newblock \emph{ESAIM: Control, Optimisation and Calculus of Variations},
  26:\penalty0 25, 2020.
\newblock \doi{10.1051/cocv/2019020}.

\bibitem[Ceci and Colaneri(2012)]{cecicolaneri:ks}
C.~Ceci and K.~Colaneri.
\newblock Nonlinear filtering for jump diffusion observations.
\newblock \emph{Adv. in Appl. Probab.}, 44\penalty0 (3):\penalty0 678--701,
  2012.
\newblock \doi{10.1239/aap/1346955260}.

\bibitem[Ceci and Colaneri(2014)]{cecicolaneri:zakai}
C.~Ceci and K.~Colaneri.
\newblock The {Z}akai equation of nonlinear filtering for jump-diffusion
  observations: existence and uniqueness.
\newblock \emph{Appl. Math. Optim.}, 69\penalty0 (1):\penalty0 47--82, 2014.
\newblock \doi{10.1007/s00245-013-9217-1}.

\bibitem[Ceci and Gerardi(2000)]{ceci2000filtering}
C.~Ceci and A.~Gerardi.
\newblock Filtering of a markov jump process with counting observations.
\newblock \emph{Applied Mathematics and Optimization}, 42\penalty0
  (1):\penalty0 1--18, 2000.

\bibitem[Ceci and Gerardi(2001)]{ceci2001nonlinear}
C.~Ceci and A.~Gerardi.
\newblock Nonlinear filtering equation of a jump process with counting
  observations.
\newblock \emph{Acta Applicandae Mathematica}, 66\penalty0 (2):\penalty0
  139--154, 2001.

\bibitem[Cohen and Elliott(2015)]{cohen:stochcalculus}
S.~N. Cohen and R.~J. Elliott.
\newblock \emph{Stochastic calculus and applications}.
\newblock Probability and its Applications. Springer, Cham, second edition,
  2015.

\bibitem[Confortola and Fuhrman(2013)]{confortola:filt}
F.~Confortola and M.~Fuhrman.
\newblock Filtering of continuous-time {M}arkov chains with noise-free
  observation and applications.
\newblock \emph{Stochastics An International Journal of Probability and
  Stochastic Processes}, 85\penalty0 (2):\penalty0 216--251, 2013.

\bibitem[Crisan and Rozovski\u{\i}(2011)]{crisanrozovski2011}
D.~Crisan and B.~Rozovski\u{\i}, editors.
\newblock \emph{The {O}xford handbook of nonlinear filtering}.
\newblock Oxford University Press, Oxford, 2011.

\bibitem[De~Angelis(2020)]{DeAngelis}
T.~De~Angelis.
\newblock Optimal dividends with partial information and stopping of a
  degenerate reflecting diffusion.
\newblock \emph{Fin. Stoch.}, 24\penalty0 (1):\penalty0 71--123, 2020.

\bibitem[D\'ecamps and Villeneuve(2020)]{DecampsVilleneuve}
J.-P. D\'ecamps and S.~Villeneuve.
\newblock Dynamics of cash holdings, learning about profitability, and access
  to the market.
\newblock Preprint,
  \href{https://www.tse-fr.eu/publications/dynamics-cash-holdings-learning-about-profitability-and-access-market}{TSE
  Working Paper, n. 19-1046}, 2020.

\bibitem[Dellacherie and Meyer(1982)]{dellacheriemeyer:B}
C.~Dellacherie and P.-A. Meyer.
\newblock \emph{Probabilities and potential. {B}}, volume~72 of
  \emph{North-Holland Mathematics Studies}.
\newblock North-Holland Publishing Co., Amsterdam, 1982.
\newblock Theory of martingales, Translated from the French by J. P. Wilson.

\bibitem[El~Karoui and Karatzas(1988)]{elkarouikaratzas}
N.~El~Karoui and I.~Karatzas.
\newblock Probabilistic aspects of finite-fuel, reflected follower problems.
\newblock \emph{Acta Appl. Math.}, 11:\penalty0 223--258, 1988.

\bibitem[Ethier and Kurtz(1986)]{ethierkurtz86:markov}
S.~N. Ethier and T.~G. Kurtz.
\newblock \emph{Markov processes}.
\newblock Wiley Series in Probability and Mathematical Statistics: Probability
  and Mathematical Statistics. John Wiley \& Sons, Inc., New York, 1986.
\newblock Characterization and convergence.

\bibitem[Fabbri et~al.(2017)Fabbri, Gozzi, and Swiech]{fabbri:soc}
G.~Fabbri, F.~Gozzi, and A.~Swiech.
\newblock \emph{Stochastic optimal control in infinite dimension}, volume~82 of
  \emph{Probability Theory and Stochastic Modelling}.
\newblock Springer, Cham, 2017.
\newblock Dynamic programming and HJB equations, With a contribution by Marco
  Fuhrman and Gianmario Tessitore.

\bibitem[Federico et~al.(2021)Federico, Ferrari, and
  Rodosthenous]{Federico-etal}
S.~Federico, G.~Ferrari, and N.~Rodosthenous.
\newblock Two-sided singular control of an inventory with unknown demand trend.
\newblock Preprint, \href{https://arxiv.org/abs/2102.11555}{arXiv:2102.11555},
  2021.

\bibitem[Grigelionis and {M}ikulevicius(2011)]{grigelionis2011}
{B}. Grigelionis and {R}. {M}ikulevicius.
\newblock Nonlinear filtering equations for processes with jumps.
\newblock In {D}. Crisan and {B}. {R}ozovskii, editors, \emph{The Oxford
  Handbook of Nonlinear Filtering}. Oxford University Press, 2011.

\bibitem[Harrison and Taksar(1983)]{HarrisonTaksar}
J.M. Harrison and M.I. Taksar.
\newblock Instantaneous control of brownian motion.
\newblock \emph{Math. Oper. Res.}, 8\penalty0 (3):\penalty0 439--453, 1983.

\bibitem[Jacod and Shiryaev(2003)]{jacod2013:limit}
J.~Jacod and A.~N. Shiryaev.
\newblock \emph{Limit theorems for stochastic processes}, volume 288 of
  \emph{Grundlehren der Mathematischen Wissenschaften}.
\newblock Springer-Verlag, Berlin, second edition, 2003.
\newblock \doi{10.1007/978-3-662-05265-5}.

\bibitem[Kallianpur(1980)]{kallianpur1980:stochfilt}
G.~Kallianpur.
\newblock \emph{Stochastic filtering theory}, volume~13 of \emph{Applications
  of Mathematics}.
\newblock Springer-Verlag, New York-Berlin, 1980.

\bibitem[Karatzas et~al.(2000)Karatzas, Ocone, Wang, and Zervos]{ocone-etal}
I.~Karatzas, D.~Ocone, H.~Wang, and M.~Zervos.
\newblock Finite-fuel singular control with discretionary stopping.
\newblock \emph{Stoch. Stoch. Rep.}, 71\penalty0 (1-2):\penalty0 1--50, 2000.

\bibitem[Kurtz and Nappo(2011)]{kurtznappo2011:FMP}
T.~G. Kurtz and G.~Nappo.
\newblock The filtered martingale problem.
\newblock In \emph{The {O}xford handbook of nonlinear filtering}, pages
  129--165. Oxford Univ. Press, Oxford, 2011.

\bibitem[Kurtz and Stockbridge(2001)]{kurtzstockbridge2001:MP}
T.~G. Kurtz and R.~H. Stockbridge.
\newblock Stationary solutions and forward equations for controlled and
  singular martingale problems.
\newblock \emph{Electron. J. Probab.}, 6:\penalty0 no. 17, 52, 2001.
\newblock \doi{10.1214/EJP.v6-90}.

\bibitem[Kurtz and Xiong(1999)]{kurtzxiong1999:SPDEs}
T.~G. Kurtz and J.~Xiong.
\newblock Particle representations for a class of nonlinear {SPDE}s.
\newblock \emph{Stochastic Process. Appl.}, 83\penalty0 (1):\penalty0 103--126,
  1999.
\newblock \doi{10.1016/S0304-4149(99)00024-1}.

\bibitem[Kurtz and Ocone(1988)]{kurtzocone88:filt}
T.G. Kurtz and D.~Ocone.
\newblock Unique characterization of condition distribution in nonlinear
  filtering.
\newblock \emph{Ann. Probab.}, 16:\penalty0 80--107, 1988.

\bibitem[Liptser and Shiryaev(2001)]{liptsershiryaev2001:statistics}
R.~S. Liptser and A.~N. Shiryaev.
\newblock \emph{Statistics of random processes. {I}}, volume~5 of
  \emph{Applications of Mathematics (New York)}.
\newblock Springer-Verlag, Berlin, expanded edition, 2001.
\newblock General theory, Translated from the 1974 Russian original by A. B.
  Aries, Stochastic Modelling and Applied Probability.

\bibitem[Lucic and Heunis(2001)]{lucicheunis2001:uniqueness}
V.~M. Lucic and A.~J. Heunis.
\newblock On uniqueness of solutions for the stochastic differential equations
  of nonlinear filtering.
\newblock \emph{Ann. Appl. Probab.}, 11\penalty0 (1):\penalty0 182--209, 2001.

\bibitem[Protter(2004)]{protter2004}
P.~E. Protter.
\newblock \emph{Stochastic integration and differential equations}, volume~21
  of \emph{Applications of Mathematics (New York)}.
\newblock Springer-Verlag, Berlin, second edition, 2004.
\newblock Stochastic Modelling and Applied Probability.

\bibitem[Reppen et~al.(2020)Reppen, Rochet, and Soner]{Reppen-etal}
M.~Reppen, J.-C. Rochet, and M.~Soner.
\newblock Optimal dividend policies with random profitability.
\newblock \emph{Math. Fin.}, 30\penalty0 (1):\penalty0 228--259, 2020.

\bibitem[Riedel and Su(2011)]{RiedelSu}
F.~Riedel and X.~Su.
\newblock On irreversible investment.
\newblock \emph{Fin. Stoch.}, 15\penalty0 (4):\penalty0 607--633, 2011.

\bibitem[Xiong(2008)]{xiong:intrtostochfiltth}
J.~Xiong.
\newblock \emph{An Introduction to Stochastic Filtering Theory}.
\newblock Oxford University Press, New York, 2008.

\end{thebibliography}
\end{document}